\documentclass[reqno]{amsart}
 \pdfoutput=1
\usepackage{amssymb}
\usepackage{graphicx}

\newtheorem{theorem}{Theorem}[section]
\newtheorem{proposition}[theorem]{Proposition}
\newtheorem{lemma}[theorem]{Lemma}
\newtheorem{corollary}[theorem]{Corollary}

\theoremstyle{remark}
\newtheorem{remark}[theorem]{Remark}

\numberwithin{equation}{section}

\begin{document}

\title[Affine Hecke algebras and Macdonald spherical functions]
{Unitary representations of affine Hecke algebras related to Macdonald spherical functions}

\author{J.F.  van Diejen}
\address{
Instituto de Matem\'atica y F\'{\i}sica, Universidad de Talca,
Casilla 747, Talca, Chile}
\email{diejen@inst-mat.utalca.cl}

\author{E. Emsiz}
\address{
Facultad de Matem\'aticas, Pontificia Universidad Cat\'olica de Chile,
Casilla 306, Correo 22, Santiago, Chile}
\email{eemsiz@mat.puc.cl}


\thanks{Work supported in part by the {\em Fondo Nacional de Desarrollo
Cient\'{\i}fico y Tecnol\'ogico (FONDECYT)} Grants \# 1090118,
11100315 and by the {\em Anillo ACT56 `Reticulados y Simetr\'{\i}as'}
financed by the  {\em Comisi\'on Nacional de Investigaci\'on
Cient\'{\i}fica y Tecnol\'ogica (CONICYT)}}

\date{January 2012}

\begin{abstract}
For any reduced crystallographic root system, we introduce a unitary representation  of
the (extended) affine Hecke algebra given by discrete difference-reflection operators acting in a Hilbert space of complex functions on the weight lattice. It is shown that the action of the center under this representation is diagonal on the basis of Macdonald spherical functions. As an application, we compute an explicit Pieri formula for these spherical functions.
\end{abstract}

\maketitle

\section{Introduction}\label{sec1}
It is well-known that Macdonald's spherical functions (on $p$-adic symmetric spaces)---also referred to as generalized Hall-Littlewood polynomials associated with root systems---are intimately connected with the theory of affine Hecke algebras \cite{mac:spherical1,mac:spherical2,nel-ram:kostka}. In a nutshell, the Macdonald spherical functions form a canonical basis of the spherical subalgebra of the affine Hecke algebra obtained from a monomial basis via the so-called Satake isomorphism. For an overview of these and many other facts concerning Macdonald's spherical functions and their relations with affine Hecke algebras we refer the reader to the comprehensive survey \cite{nel-ram:kostka}.

The interplay between affine Hecke algebras and Macdonald spherical functions has proven very fruitful. For instance, affine Hecke algebras turn out to be  instrumental in obtaining explicit combinatorial formulas for the monomial expansion and for the structure constants (or Littlewood-Richardson type coefficients) of the Macdonald spherical functions \cite{par:buildings,ram:alcove,sch:galeries}.  Reversely, properties of Macdonald's spherical functions---in particular Macdonald's orthogonality relations
and the (generalized) Kostka-Foulkes coefficients describing the transition between Macdonald's spherical functions and the basis of Weyl characters---are fundamental, respectively, in the harmonic analysis of the affine Hecke algebra \cite{opd:spectral} and for the explicit computation of the Kazhdan-Lusztig basis for the spherical Hecke algebra  \cite{nel-ram:kostka,kno:kazhdan-lusztig}.

The present paper studies the properties of a concrete difference-reflection representation of the affine Hecke algebra and its relations to the theory of Macdonald's spherical functions.
Specifically, we introduce an explicit unitary representation  of
the (extended) affine Hecke algebra in terms of discrete difference-reflection operators acting in a Hilbert space of
complex functions on the weight lattice and show that the action of its center under this
representation is diagonal on the basis of Macdonald spherical functions. The main technical difficulty in the diagonalization proof is the verification of  intertwining relations between our difference-reflection representation and a second auxiliary representation (in terms of discrete integral-reflection operators) that is dual to the standard induced polynomial representation of the affine Hecke algebra. As an application, we compute an explicit Pieri formula for the Macdonald spherical functions generalizing the Pieri formula for the Hall-Littlewood polynomials due to Morris (from root systems of type $A$ to arbitrary type) \cite{mor:note}.

Our results provide a link interpolating between the Hecke-algebraic techniques developed in the spectral theory of quantum integrable particle systems \cite{hec-opd:yang,ems-opd-sto:periodic} and those employed in Macdonald's theory of symmetric orthogonal polynomials \cite{mac:affine,che:double}.
Indeed, it is known that the Macdonald spherical functions tend in an appropriate continuum limit to the eigenfunctions of the Laplacian perturbed by a delta potential supported on (the hyperplanes of) the corresponding root system \cite{hec-opd:yang,die:plancherel}. In this limiting situation the role of the affine Hecke algebra is played by the Drinfeld-Lusztig graded (or degenerate) affine Hecke algebra \cite{hec-opd:yang,ems-opd-sto:periodic}. Specifically, our difference-reflection representation gets replaced by a representation of the graded affine Hecke algebra built of Dunkl-type differential-reflection operators, the discrete integral-reflection (or polynomial) representation gets replaced by a representation of the graded affine Hecke algebra in terms of Gutkin-Sutherland continuous integral-reflection operators, and the intertwining operator relating both representations is given by the Gutkin-Sutherland propagation operator \cite{gut-sut:completely,gut:integrable,hec-opd:yang,ems-opd-sto:periodic}. From this perspective, the present paper lifts this construction to the level of the affine Hecke algebra corresponding to the Macdonald spherical functions.
On the other hand, it is well-known that the Macdonald spherical functions are limiting cases of the celebrated Macdonald polynomials (corresponding to $q\to 0$) \cite{mac:orthogonal}. The Macdonald polynomials in turn diagonalize a commuting algebra of Macdonald difference operators that can be constructed by means of Cherednik's extension of the polynomial representation of the affine Hecke algebra to the level of the double affine Hecke algebra \cite{mac:affine,che:double}. From this perspective, the difference-reflection representation of the affine Hecke algebra studied here provides the corresponding concrete construction of the commuting algebra of discrete difference operators  that is diagonalized by the Macdonald spherical functions (and isomorphic to the Weyl-group invariant part  of the group algebra
over the weight lattice).

The paper is organized as follows. In Section \ref{sec2} notational preliminaries concerning affine Weyl groups and affine Hecke algebras are recalled.
In Section \ref{sec3}
our main representation of the affine Hecke algebra in terms of
difference-reflection operators is introduced.
The auxiliary representation of the affine Hecke algebra by integral-reflection operators  and its relation to the standard polynomial representation are described in Section \ref{sec4}.
Section \ref{sec5} introduces an intertwining operator between the difference-reflection representation and the auxiliary integral-reflection representation, which is then used to show that the action of the center under the
difference-reflection representation is diagonal on the Macdonald spherical functions.
In Section \ref{sec6}  the appropriate Hilbert space structure is provided for which the difference-reflection representation is unitary. From this viewpoint the Macdonald spherical function constitutes the kernel of the Fourier transform---between the Weyl-group invariant sector of this Hilbert space and a closure of the Weyl-group invariant  part of the group algebra of the weight lattice---diagonalizing the action of the center under our difference-reflection representation. Finally, in Section \ref{sec7} we use the difference-reflection representation
to compute the explicit Pieri formula for the Macdonald spherical functions.
Some technical details pertaining to the proof of the braid relations in Section \ref{sec3}
and the intertwining relations in Section \ref{sec5} are relegated to
Appendices \ref{appA} and \ref{appB}, respectively. Moreover, in
Appendix \ref{appC} a few illuminating explicit formulas are collected describing our principal objects of study
in the important special case of a root system of type $A_{N-1}$ (i.e., with the Weyl group being equal to the permutation
group $S_N$).

\section{Preliminaries}\label{sec2}
This section sets up the notation for affine Weyl groups and their Hecke algebras and recalls briefly some basic properties. A more thorough discussion with proofs can be found e.g. in the standard sources \cite{bou:groupes,mac:affine}.

\subsection{Affine Weyl group}
Let $R$ be a crystallographic root system spanning a real (finite-dimensional) Euclidean vector space
$V$ with inner product $\langle \cdot ,\cdot\rangle$. Throughout it will be assumed that $R$ is both irreducible and reduced (unless explicitly stated otherwise). Following standard conventions,   the dual root system is denoted by $R^\vee:=\{\alpha^\vee\mid \alpha\in R\}$ with $\alpha^\vee:=2\alpha /\langle \alpha ,\alpha\rangle$, the weight lattice by $P:=\{ \lambda\in V\mid \langle \lambda ,\alpha^\vee\rangle \in \mathbb{Z},
\forall \alpha \in R\}$, and for a (fixed) choice of positive roots $R^+$ we write
$P^+:=\{ \lambda\in P\mid \langle \lambda ,\alpha^\vee\rangle \geq 0,
\forall \alpha \in R^+\}$
for the corresponding cone of dominant weights and $C:= \{ x\in V \mid  \langle x,\alpha^\vee\rangle >0, \forall \alpha \in R^+\} $ and $A:= \{ x\in V \mid 0< \langle x,\alpha^\vee\rangle < 1, \forall \alpha \in R^+\} $ for the dominant Weyl chamber and Weyl alcove, respectively.

For $\alpha \in R^+$ and $k\in \mathbb{Z}$  let $s_{\alpha,k}:V\to V$ be the orthogonal
reflection across the hyperplane $V_{\alpha ,k}:=\{ x\in V \mid  \langle x ,\alpha^\vee\rangle = k \}$
and for $\lambda \in P$ let  $t_\lambda :V\to V$ be the translation of the form
$t_\lambda(x):=x+\lambda$ ($x\in V$).
The (finite) {\em Weyl group} generated by the reflections $s_{\alpha,0}$, $\alpha\in R^+$ is denoted by $W_0$ and we write $W$ for
the (extended) {\em affine Weyl group} generated by the elements of
$W_0$ and the translations $t_\lambda$, $\lambda\in P$.
The length of a group element $w\in W$ is defined as the cardinality
$\ell (w):= |S(w)|$ of the set
$S(w):=\{ V_{\alpha ,k} \mid
V_{\alpha ,k}\ \text{separates}\ A \ \text{and}\  wA \}$.  (We say that a hyperplane $V_{\alpha ,k}$ separates two (subsets of) points in $V$ if these are contained in distinct connected components of $V\setminus V_{\alpha ,k}$.)
A useful explicit formula to compute the lengths of (affine) Weyl group elements  is given by
\begin{equation}\label{length}
\ell (v t_\lambda ) = \sum_{\alpha\in R^+} | \langle \lambda ,\alpha^\vee\rangle +\chi (v\alpha)|
\quad (v\in W_0,\lambda\in P ),
\end{equation}
where $\chi:R\to \{ 0,1\}$ represents the characteristic function of $R^-:=R\setminus R^+$  (so, in particular, for $v\in W_0$ and
$\lambda,\mu\in P^+$ one has that $\ell (v t_\lambda )=\ell (v)+\ell (t_\lambda )$ and that $\ell (t_{v\mu})=\ell (t_\mu )$, $\ell (t_{\lambda +\mu})=\ell (t_\lambda )+\ell (t_\mu )$.)

Let $\alpha_1,\ldots,\alpha_n$ ($n:=\text{rank}(R)$) be the basis of simple roots for $R^+$ and let $\alpha_0$ be the positive root such that $\alpha_0^\vee$ is the highest root of $R^\vee$ (with respect to $(R^+)^\vee$).  We set $s_0:=s_{\alpha_0,1}$ and $s_j:=s_{\alpha_j,0}$ for $j=1,\ldots ,n$.
The finite Weyl group $W_0$ is generated by the reflections  across the boundary hyperplanes of $C$: $s_1,\ldots ,s_n$, and the affine Weyl group $W$ is generated by
 the (finite, Abelian) subgroup  of elements of length zero $\Omega :=\{ u\in W\mid uA=A\}$
and the reflections across the boundary hyperplanes of $A$:
$s_0,\ldots ,s_n$.

It is instructive to detail the algebraic structure of these presentations of $W_0$ and $W$ somewhat more explicitly.
Let $V_j$ denote the hyperplane fixed by $s_j$ ($j=0,\ldots ,n$).
The finite Weyl group $W_0$ amounts to the group generated by $s_1,\ldots ,s_n$ subject to the relations
\begin{equation}\label{W-rel1}
(s_js_k)^{m_{jk}}=1,
\end{equation}
with $\pi/m_{jk}$ being the angle between $V_j$ and $V_k$ if $j\neq k$
(so, in particular, $m_{kj}=m_{jk}$) and $m_{jk}=1$ when $j=k$.
To characterize $\Omega$ it is convenient to associate with $\lambda\in P$ the affine Weyl group element
$u_\lambda:= t_\lambda v_\lambda^{-1}$, where
$v_\lambda$ refers to the shortest element of $W_0$ mapping $\lambda$ to the closure of the antidominant Weyl chamber $-C$ (which implies that $u_\lambda$ is the shortest element of the coset $t_\lambda W_0$ and
$\ell(t_\lambda)=\ell (u_\lambda )+\ell (v_\lambda )$).
Upon setting $u_0:=1$ and $u_j:=u_{\omega_j}$ for $j=1,\ldots ,n$, where $\omega_1,\ldots ,\omega_n$
denote the basis of the fundamental weights, one has explicitly
\begin{equation}\label{finite-group}
\Omega = \{ u_j \mid j=0\ \text{or}\ \langle \omega_j,\alpha_0^\vee \rangle =1 \} .
\end{equation}
It is clear from the definition that the elements of $\Omega$ permute the hyperplanes
$V_0,\ldots ,V_n$. Furthermore, for $u\in\Omega$ with $uV_j=V_k$ one has that
\begin{equation}\label{W-rel2}
uu_j=u_ju=u_k \quad (u_j\in \Omega )\quad\text{and}\quad us_j=s_ku   \quad  (j=0,\ldots ,n).
\end{equation}
The affine Weyl group $W$ can now be characterized as the group generated by $s_0,\ldots ,s_n$ and the elements  $u\in\Omega$ \eqref{finite-group} subject to the relations
\eqref{W-rel1}, \eqref{W-rel2} (with the additional caveat that in the pathological case $n=1$ the order
$m_{10}=m_{01}=\infty$).

\subsection{Affine Hecke algebra}
Let $q:W\to \mathbb{R}\setminus \{ 0\}$ be a length multiplicative function, viz. (i)
$q_{ww^\prime}= q_w q_{w^\prime} $ if $\ell(ww^\prime )=\ell(w)+\ell(w^\prime ) $ and (ii)
$q_w=1$ if $ \ell(w)=0$. This implies that $q_{s_j}$ depends only on the conjugacy class of $s_j$
$(j=0,\ldots ,n)$, whence the value of $q_w$ is determined by  the number of
reflections (in the short roots and in the long roots, respectively) appearing in a reduced expression $w=us_{j_1}\cdots s_{j_\ell} $ (with $u\in\Omega$ and $\ell=\ell (w)$).
Following customary habits, the multiplicity function associated with the length multiplicative function will also be denoted by $q$.
This is the function $q:R^+\times \mathbb{Z}\to \mathbb{R}\setminus \{ 0\}$ such that
$q_{\alpha,k}=q_{s_j}$ if $V_{\alpha ,k}=V_j$  ($0\leq j\leq n$) and
$q_{\alpha^\prime,k^\prime}=q_{\alpha ,k}$ if $V_{\alpha^\prime ,k^\prime}=wV_{\alpha ,k}$ for some $w\in W$.
This implies that $q_{\alpha ,k}=q_{\alpha ,0}$ depends only on the length of $\alpha$.
Reversely, the length multiplicative function can be reconstructed from the multiplicity function via the formula
\begin{equation}\label{reconstruct}
q_w=\prod_{\substack{\alpha\in R^+,k\in\mathbb{Z}\\ V_{\alpha, k}\in S(w) }} q_{\alpha,k} .
\end{equation}
We will write $\mathcal{H}$ for the (extended) {\em affine Hecke algebra} associated with $W$ and $q$. This algebra can be characterized as the complex associative algebra with basis $T_w$, $w\in W$ satisfying
the {\em quadratic relations}
\begin{subequations}
\begin{equation}\label{quadratic-rel}
(T_j-q_j)(T_j+q_j^{-1})=0,\quad j=0,\ldots ,n,
\end{equation}
where $T_j:=T_{s_j}$ and $q_j:=q_{s_j}$, and
the {\em braid relations}
\begin{equation}\label{braid-rel}
T_{ww^\prime} = T_wT_{w^\prime}\quad\text{if}\
\ell(ww^\prime )=\ell(w)+\ell(w^\prime )  .
\end{equation}
\end{subequations}
The assignment $T_w\to T_w^*$ with
\begin{equation}
T_w^*:=T_{w^{-1}}
\end{equation}
extends to an antilinear anti-involution of $\mathcal{H}$ thus turning the affine Hecke algebra into an involutive or $*$-algebra.
The subalgebra of $\mathcal{H}$ spanned by the basis $T_w$, $w\in W_0$ is referred to as the finite {\em Hecke algebra} $\mathcal{H}_0$ (associated with $W_0$ and $q$).

The affine Hecke algebra $\mathcal{H}$ admits a simple presentation as the algebra generated by $T_0,\ldots ,T_n$ and $T_u$,
$u\in\Omega$ \eqref{finite-group} subject to the quadratic relations
 \eqref{quadratic-rel} (cf. Eq. \eqref{W-rel1} with $j=k$), the braid relations
 \begin{equation}\label{T-rel1}
\underbrace{T_jT_kT_j\cdots}_{m_{jk}\ {\rm factors}}
=\underbrace{T_kT_jT_k\cdots}_{m_{jk}\ {\rm factors}},\qquad j\neq k
\end{equation}
(cf. Eq. \eqref{W-rel1} with $j\neq k$), and the relations
\begin{equation}\label{T-rel2}
T_uT_{u_j}=T_{uu_j}=T_{u_k}\quad (u_j\in\Omega )\quad\text{and}\quad T_uT_j=T_kT_u   \quad (j=0,\ldots ,n),
\end{equation}
with $V_k =uV_j$ (cf. Eq. \eqref{W-rel2}).
The finite Hecke algebra $\mathcal{H}_0$ in turn amounts to the (sub)algebra generated by $T_1,\ldots ,T_n$
subject to the quadratic relations \eqref{quadratic-rel} and the braid relations \eqref{T-rel1}.

For $\lambda\in P$, the element
\begin{equation}
Y^\lambda := T_{t_\mu}T_{t_{\nu}}^{-1}\quad \text{with} \  \mu, \nu\in P^+\ \text{such that }\ \lambda = \mu-\nu
\end{equation}
is well-defined in the sense that it does not depend on the particular choice of the
decomposition of $\lambda$ as a difference of dominant weights $\mu$ and $\nu$.
Furthermore, the elements $Y^\lambda$, $\lambda\in P$ form a basis of a subalgebra of $\mathcal{H}$ isomorphic to the group algebra  of the weight lattice $\mathbb{C}[P]$:
\begin{subequations}
\begin{equation}\label{group-algebra-rel}
 Y^\lambda Y^\mu = Y^{\lambda +\mu}    \quad (\lambda,\mu \in P)\quad \text{and}\quad  Y^0=1,
\end{equation}
satisfying in addition the relations
\begin{equation}\label{cross-rel1}
T_jY^\lambda - Y^{s_j\lambda}T_j=(q_j-q_j^{-1})\frac{Y^\lambda-Y^{s_j\lambda}}{1-Y^{-\alpha_j}},
\quad (j=1,\ldots ,n).
\end{equation}
\end{subequations}

The elements $T_wY^\lambda$, $w\in W_0$, $\lambda\in P$ constitute a basis of $\mathcal{H}$, which gives rise to a second (very useful) presentation of the affine Hecke algebra (due to Bernstein, Lusztig, and Zelevinsky) as the algebra generated by $T_1,\ldots ,T_n$ and $Y^\lambda$, $\lambda\in P$ subject to the relations \eqref{quadratic-rel}, \eqref{T-rel1}, \eqref{group-algebra-rel} and
\begin{equation}\label{cross-rel2}
\begin{array}{lc}
T_jY^\lambda = Y^\lambda T_j  & \text{if}\ \langle \lambda ,\alpha_j^\vee\rangle =0 ,\\ [1ex]
T_jY^\lambda = Y^{s_j\lambda} T_j +  (q_j-q_j^{-1})Y^\lambda & \text{if}\ \langle \lambda ,\alpha_j^\vee\rangle =1
\end{array}
\end{equation}
 (cf. Eq. \eqref{cross-rel1} with
$\lambda\in V_{\alpha_j,0}\cup V_{\alpha_j,1}$). In other words, the affine Hecke algebra $
\mathcal{H}$ is a merger of the finite Hecke algebra $\mathcal{H}_0$ and the group algebra $\mathbb{C}[P]$ with the cross relations \eqref{cross-rel2}.

It can be seen with the aid of the latter presentation that the center $\mathcal{Z}$ of $\mathcal{H}$ is spanned
by
\begin{equation}
m_\lambda (Y):= \sum_{\mu\in W_0\lambda} Y^\mu,\quad \lambda\in P^+
\end{equation}
(and thus isomorphic to the $W_0$-invariant part $\mathbb{C}[P]^{W_0}$ of the group algebra of the weight lattice). Moreover, since
\begin{equation}
(Y^\lambda)^*=T_{w_o}Y^{-w_o\lambda}T_{w_o}^{-1}
\end{equation}
(where $w_o$ denotes the longest element of $W_0$),  one has that
\begin{equation}
m_\lambda(Y)^*=m_{\lambda^*}(Y),\quad\text{with}\quad \lambda^*:=-w_o\lambda.
\end{equation}

\section{Difference-reflection operators}\label{sec3}
In this section we introduce our main representation of the affine Hecke algebra in terms of
difference-reflection operators.

The action of the affine Weyl group on $P\subset V$ induces a representation of $W$
on the space $C(P):=\{ f\mid f:P\to\mathbb{C}\} $
\begin{equation}\label{W-action}
(wf)(\lambda ) := f(w^{-1}\lambda )\qquad (w\in W,\ \lambda\in P).
\end{equation}
We consider the following difference-reflection operators on $C(P)$
\begin{subequations}
\begin{equation}\label{Tj-operator}
\hat{T}_j:=q_j+\chi_j (s_j-1),\qquad j=0,\ldots ,n,
\end{equation}
where  $q_j$ and $\chi_j$ act by multiplication with
\begin{equation}\label{chi}
\chi_j (\lambda) :=
\begin{cases}
q_j&\text{if}\  V_j\  \text{separates}\ \lambda\ \text{and}\ A,\\
1 &\text{if}\ \lambda\in V_j, \\
q_j^{-1}&\text{otherwise} .
\end{cases}
\end{equation}
\end{subequations}

\begin{theorem}[Difference-Reflection Representation $\hat{T}(\mathcal{H})$]\label{dif-ref-rep:thm} The assignment $T_j\to\hat{T}_j$ ($j=0,\ldots ,n$) and $T_u\to u$ ($u\in\Omega$) extends (uniquely) to a representation $h\to \hat{T}(h)$ ($h\in \mathcal{H}$) of the affine Hecke algebra on $C(P)$.
\end{theorem}

Inferring this theorem amounts to verifying the relations
\begin{subequations}
\begin{eqnarray}
\label{qr} (\hat{T}_j-q_j)(\hat{T}_j+q_j^{-1})=0&& (0\leq j\leq n), \\
\label{br} \underbrace{ \hat{T}_j\hat{T}_k\hat{T}_j\cdots}_{m_{jk}\ {\rm factors}}
=
 \underbrace{ \hat{T}_k\hat{T}_j\hat{T}_k\cdots}_{m_{jk}\ {\rm factors}}
&& (0\leq j\neq k\leq n),\\
\label{cr} u\hat{T}_j=\hat{T}_k u\ \ \text{if}\  \  uV_j =V_k&&  (u\in\Omega,\ 0\leq j\leq n).
\end{eqnarray}
\end{subequations}
The  quadratic relations in Eq. \eqref{qr} follow from a short computation:
\begin{align*}
\hat{T}_j^2&=q_j^2+(2q_j\chi_j-1-\chi_j^2)(s_j-1) \\
&=q_j^2+(q_j-q_j^{-1})\chi_j(s_j-1)=(q_j-q_j^{-1})\hat{T}_j+1,
\end{align*}
where we used (in the second identity) that for $\lambda\not\in V_j$
$$2q_j\chi_j-1-\chi_j^2=(q_j-q_j^{-1})\chi_j=\begin{cases} q_j^2-1&\text{if}\ \chi_j=q_j \\
1-q_j^{-2}&\text{if}\ \chi_j=q_j^{-1} \end{cases},$$
together with the observation that for any $f\in C(P)$ the difference $(s_jf)(\lambda)-f(\lambda)$ vanishes
when $\lambda\in V_j$.
The commutation relations in Eq. \eqref{cr} are in turn immediate from the definition of $\hat{T}_j$ and the corresponding affine Weyl group relations
in Eq. \eqref{W-rel2}.
The proof of the braid relations in Eq. \eqref{br} is a bit more intricate and
hinges on two lemmas that require some additional notation.
For $x\in V$ let $W_{0,x} \subset W_0$ denote the stabilizer subgroup $\{ w\in W_0 \mid wx =x \}$. We will consider the following equivalence relation on $V$:
$x\sim y$ iff
 $W_{0,x}=W_{0,y}$ and both points
lie on the closure of the same Weyl chamber $wC$ (for some $w\in W_0$).
The finite number of equivalence classes of $V$ with respect to the relation $\sim$ are called facets and constitute the so-called Coxeter complex $\mathcal{C}$ of $W_0$.

\begin{lemma}\label{facets:lem}
Let $\hat{D}$ be an operator in $\mathbb{C}\bigl\langle\hat{T}_1,\ldots ,\hat{T}_n\bigr\rangle$ and let $\lambda ,\mu\in P$ with $\lambda\sim\mu$. Then
\begin{equation}\label{facets}
  (\hat{D}f)(\lambda)=0\quad\forall f\in C(P)
\Longrightarrow
 (\hat{D}f)(\mu)=0\quad\forall f\in C(P) .
 \end{equation}
\end{lemma}

\begin{proof}
Given $f\in C(P)$ and  $\lambda,\mu\in P$ with $\lambda\sim\mu$, pick an $\tilde{f}\in C(P)$ such that
$\tilde{f}(w\lambda )=f(w\mu )$ for all $w\in W_0$.  (Such a function $\tilde{f}$ exists, since
$w\lambda = w^\prime \lambda \Rightarrow w^{-1}w^\prime\in W_{0,\lambda} =W_{0,\mu} \Rightarrow
w\mu = w^\prime \mu$.) From the definition of the difference-reflection operators $\hat{T}_1,\ldots ,\hat{T}_n$ it is then immediate that $(\hat{D}f)(\mu )=(\hat{D}\tilde{f})(\lambda )$ (because $w\mu$ and $w\lambda$ ($w\in W_0$) cannot be separated by the hyperplanes $V_j$, $1\leq j\leq n$ as both weights lie on the same facet). Hence, the hypothesis on the LHS of Formula \eqref{facets} implies that for {\em any} $f\in C(P)$: $(\hat{D}f)(\mu )=(\hat{D}\tilde{f})(\lambda )=0$.
\end{proof}

\begin{lemma}\label{rank2-reduction:lem}
For any affine Weyl group $W$, the braid relations in Eq. \eqref{br} follow from
the braid relations corresponding to the {\em finite} Weyl groups $W_0$ associated with the (not necessarily irreducible) root systems of rank two.
\end{lemma}

\begin{proof}
Without restriction we may assume that $n\geq 2$ (as for $n=1$ there is no braid relation to check since
then $m_{01}=m_{10}=\infty$).
For any pair  $1\leq j\neq k\leq n$, the reflections  $s_j, s_k$ generate a finite Weyl group corresponding to the rank-two root subsystem $R_{jk}$ with basis $\alpha_j,\alpha_k$. The Weyl group in question acts trivially on the orthogonal complement $V_{jk}^\perp$ of $V_{jk}:=\text{Span}_{\mathbb{R}}(\alpha_j,\alpha_k)$ in $V$. It follows that the action of $\hat{T}_j$ and $\hat{T}_k$ on $C(P)$ extends to a decomposition of the form $\hat{T}_j({R_{jk}})\otimes 1$ and $\hat{T}_k({R_{jk}})\otimes 1$
on $F(P_{jk})\otimes F(V_{jk}^\perp)\supset C(P)$, where $P_{jk}$ denotes the image of the orthogonal projection of $P$ onto $V_{jk}$ (and $F(P_{jk})$,  $F(V_{jk}^\perp)$ are the spaces of complex functions on $P_{jk}$ and $V_{jk}^\perp$, respectively). Here $\hat{T}_j({R_{jk}})$ and $\hat{T}_k({R_{jk}})$ refer to the
corresponding operators on $F(P_{jk})$ associated with the simple reflections of $R_{jk}$. (Notice in this connection that $P_{jk}$ amounts to the weight lattice $P(R_{jk})$ associated with $R_{jk}$ and that the image $A_{jk}$ of the alcove $A$ under the orthogonal projection onto $V_{jk}$ is contained in the Weyl alcove $A(R_{jk})$ associated with the basis $\alpha_j,\alpha_k$ of $R_{jk}$.)
The upshot is that the braid relations
for $\hat{T}_j$ and $\hat{T}_k$ follow from the braid relations for $\hat{T}_j({R_{jk}})$ and $\hat{T}_k({R_{jk}})$.
If one of the two indices ($j$ say) takes the value $0$, then the above arguments apply verbatim upon picking for $R_{0k}$ the translated rank-two root system with basis  $-\alpha_0, \alpha_k$ relative to the origin at
$V_{0k}\cap V_0\cap V_k $ (where $V_{0k}=\text{Span}_\mathbb{R}(\alpha_0,\alpha_k)$). (Now the projection $P_{0k}$ of $P$ onto $V_{0k}$ amounts rather to the
weight lattice of the untranslated rank-two root subsystem with basis $-\alpha_0,\alpha_k$, but this is no obstacle in view of Remark \ref{extension:rem} below.)
\end{proof}

By Lemma \ref{rank2-reduction:lem}, it is sufficient to verify the braid relations
$\hat{T}_1\hat{T}_2\hat{T}_1\cdots=\hat{T}_2\hat{T}_1\hat{T}_2\cdots$ (with $m_{12}$ factors on both sides) associated with the simple reflections $s_1$ and $s_2$ for the  root systems
$A_1\times A_1$, $A_2$, $B_2$,  and $G_2$ (for which
$m_{12}=2$, $3$, $4$, and $6$, respectively). Moreover by Lemma \ref{facets:lem}---upon acting with both sides on an arbitrary lattice function in $C(P)$---it is only needed to
verify these braid relations on a finite $W_0$-invariant set of weights representing the facets
of the Coxeter complex $\mathcal{C}$.
This reduces the verification of Eq. \eqref{br} to a routine case-by-case computation that is somewhat tedious by hand for the three root systems other than  $A_1\times A_1$ (and particularly so for the root systems $B_2$ and $G_2$)
but completely straightforward to perform in all four cases with the aid of symbolic computer algebra.
To illustrate the idea of the computation in question we have outlined the details for the root system $A_2$
in Appendix \ref{appA}.

\begin{remark}\label{extension:rem}
The action of the affine Weyl group in Eq. \eqref{W-action} and the operators $\hat{T}_j$ \eqref{Tj-operator}, \eqref{chi} make
in fact sense on the space $F(\mathcal{A})$ of complex functions on
the Coxeter complex $\mathcal{A}$ of the {affine} Weyl group (which may also be seen as the space of functions on $V$ that are piecewise constant on the {affine} facets). (Here the affine facets are the equivalence classes of $V$ with
points being equivalent if they belong to the closure of the same Weyl alcove $wA$ ($w\in W$) and have the same
stabilizer inside the affine Weyl group.)
The space $C(P)$ can be naturally embedded into $F(\mathcal{A})$ as the space of functions with support in the affine facets containing a weight (since points differing by a nonzero weight necessarily belong to distinct affine facets).
With this extension of the domain, the Hecke-algebra relations in Eqs. \eqref{qr}-\eqref{cr} remain valid. Indeed, Lemma \ref{facets:lem} and its proof generalize verbatim from $P$ to $\mathcal{A}$. In other words,
the representation  in Theorem \ref{dif-ref-rep:thm} extends naturally to a representation of the affine Hecke algebra on the space $F(\mathcal{A})$.
\end{remark}

\section{Integral-reflection operators}\label{sec4}
In this section we describe the auxiliary representation of the affine Hecke algebra in terms of integral-reflection operators. The representation in question is dual to a standard polynomial representation of the affine Hecke algebra on the group algebra of the weight lattice.

We consider the following integral-reflection operators on $C(P)$ associated with the simple reflections $s_1,\ldots ,s_n$:
\begin{subequations}
\begin{equation}\label{int-op1}
I_j := q_js_j+(q_j-q_j^{-1})J_j ,\quad j=1,\ldots ,n,
\end{equation}
where $J_j:C(P)\to C(P)$ denotes a discrete integral operator which---grosso modo---integrates the lattice function
$f(\lambda )$ over the $\alpha_j$-string from $\lambda$ to $s_j\lambda$:
\begin{eqnarray}\label{int-op2}
\lefteqn{(J_jf)(\lambda ) :=} && \\
&& \begin{cases}
-f(\lambda-\alpha_j)-
f(\lambda-2\alpha_j)-\cdots -f(s_j\lambda) &\text{if}\ \langle \lambda ,\alpha_j^\vee\rangle > 0 ,\\
0&\text{if}\  \langle \lambda ,\alpha_j^\vee\rangle = 0 ,\\
f(\lambda)+f(\lambda+\alpha_j)+\cdots +f(s_j\lambda-\alpha_j)
&\text{if}\  \langle \lambda ,\alpha_j^\vee\rangle < 0 .
\end{cases} \nonumber
\end{eqnarray}
\end{subequations}

\begin{proposition}[Integral-Reflection Representation $I(\mathcal{H})$]\label{int-ref-rep:prp}
The assignment $T_j\to I_j$ ($j=1,\ldots ,n$) and $Y^\lambda\to t_\lambda$ ($\lambda\in P$) extends (uniquely)
to a representation $h\to I(h)$ ($h\in \mathcal{H}$) of the affine Hecke algebra on $C(P)$.
\end{proposition}
In the remainder of this section the proposition is proved by exploiting that $I(\mathcal{H})$ may be seen
as the dual of a standard representation of the affine Hecke algebra in terms of Demazure-Lusztig operators.

Let us denote by  $e^\lambda$, $\lambda\in P$ the standard basis of the group algebra $\mathbb{C}[P]$ (so $e^\lambda e^\mu = e^{\lambda+\mu}$ and $e^0=1$) and consider the following nondegenerate sesquilinear pairing $(\cdot, \cdot ):  C(P)\times  \mathbb{C}[P]\to\mathbb{C}$
\begin{equation}\label{pairing}
(f, p):=(\bar{ p} f)(0 )\qquad ( f\in C(P),p\in \mathbb{C}[P]),
\end{equation}
where $\bar{p}$ refers to the complex conjugate $\sum_\lambda \bar{c}_\lambda e^\lambda$ of $p=\sum_\lambda c_\lambda e^\lambda$ ($c_\lambda\in\mathbb{C}$), and the action of  $ \mathbb{C}[P]$ on $f$ is determined by $e^\lambda f :=t_\lambda f$ (so, in particular, $(f,e^\lambda )=(t_\lambda f)(0)=f(-\lambda )$). We will use the notational convention $(p,f):=\overline{(f,p)}$.
The action of $W$ on $P$ lifts to an action of the affine Weyl group on $\mathbb{C}[P]$ via $we^\lambda:=e^{w\lambda}$ ($w\in W$, $\lambda\in P$). Notice that with these conventions
 $(v t_\lambda f,p)=(f ,t_\lambda v^{-1} p)$ ($v\in W_0$, $\lambda\in P$, $f\in C(P)$, $p\in\mathbb{C}[P]$), i.e.
the action of $W_0$ is `unitary'  and the action of $P$ is `symmetric' with respect to the above pairing.

It is well-known (cf. e.g.  Ref. \cite{mac:affine}) that the trivial one-dimensional representation $T_j\to q_j$ ($j=1,\ldots ,n$) of $\mathcal{H}_0$ on $\mathbb{C}$ immediately induces a representation of the finite Hecke algebra on the group algebra through the relations in Eq. \eqref{cross-rel1}. Indeed, the latter representation $h\to \check{T}(h)$
of $\mathcal{H}_0$ on  $\mathbb{C}[P]$
is generated by the Demazure-Lusztig operators:
\begin{equation}\label{demazure}
\check{T}_j:=q_j s_j +(q_j-q_j^{-1}) (1-e^{-\alpha_j})^{-1}(1-s_j), \qquad j=1,\ldots ,n.
\end{equation}

Proposition \ref{int-ref-rep:prp} is now a direct consequence of the two subsequent lemmas and the Bernstein-Lusztig-Zelevinsky presentation of the affine Hecke algebra with the relations in Eq. \eqref{cross-rel2}.

\begin{lemma}\label{finite-rep:lem}
The assignment $T_j\to I_j$ ($j=1,\ldots ,n$) extends (uniquely)
to a representation $h\to I(h)$ ($h\in \mathcal{H}_0$) of the finite Hecke algebra on $C(P)$, i.e.
\begin{subequations}
\begin{eqnarray}
(I_j-q_j)(I_j+q_j^{-1})=0&& (1\leq j\leq n), \\
\underbrace{ I_jI_kI_j\cdots}_{m_{jk}\ {\rm factors}}
=
 \underbrace{ I_kI_jI_k\cdots}_{m_{jk}\ {\rm factors}} &&  (1\leq j\neq k \leq n) .
\end{eqnarray}
\end{subequations}
\end{lemma}

\begin{proof}
By acting with the Demazure-Lusztig operator $\check{T}_j $ \eqref{demazure}
on the basis element $e^\lambda$ it is seen that
\begin{align*}
\check{T}_je^\lambda &=q_j e^{s_j\lambda}+
(q_j-q_j^{-1})\frac{e^\lambda -e^{\lambda-\langle\lambda,\alpha_j^\vee\rangle\alpha_j}}{1-e^{-\alpha_j}} \\
&=q_j e^{s_j\lambda}+(q_j-q_j^{-1})\times
\begin{cases}
e^\lambda + e^{\lambda-\alpha_j}+\cdots + e^{s_j\lambda+\alpha_j}
&\text{if}\ \langle \lambda,\alpha_j^\vee\rangle > 0 ,\\
0 &\text{if}\  \langle \lambda,\alpha_j^\vee\rangle = 0,\\
-e^{\lambda+\alpha_j} - e^{\lambda+2\alpha_j}-\cdots - e^{s_j\lambda}
&\text{if}\ \langle \lambda,\alpha_j^\vee\rangle < 0 ,
\end{cases}
\end{align*}
whence  $(I_jf,e^\lambda)=(f,\check{T}_j e^\lambda )$ ($f\in C(P)$, $\lambda\in P$).
The quadratic relations and braid relations for $I_1,\ldots ,I_n$ thus follow from those for
$\check{T}_1,\ldots ,\check{T}_n$ (and $( I(h)f,p)=(f,\check{T}({h^*})p)$, $h\in\mathcal{H}_0$, $f\in C(P)$, $p\in\mathbb{C}[P]$).
\end{proof}

\begin{lemma}
The operators $I_j$ ($j=1,\ldots ,n$) and $t_\lambda $ ($\lambda\in P$) on $C(P)$ satisfy the cross relations
\begin{equation}
\begin{array}{lc}
I_j t_\lambda = t_\lambda I_j  & \text{if}\ \langle \lambda ,\alpha_j^\vee\rangle =0 ,\\ [1ex]
I_j t_\lambda = t_{s_j\lambda} I_j +  (q_j-q_j^{-1})t_\lambda & \text{if}\ \langle \lambda ,\alpha_j^\vee\rangle =1 .
\end{array}
\end{equation}
\end{lemma}

\begin{proof}
Since $s_jt_\lambda=t_{s_j\lambda}s_j$, it is sufficient to infer that
$J_jt_\lambda=t_\lambda J_j$ if $\langle \lambda ,\alpha_j^\vee\rangle =0$ and that
$J_jt_\lambda=t_{s_j\lambda} J_j+t_\lambda $ if $\langle \lambda ,\alpha_j^\vee\rangle =1$.
Both identities are seen to hold manifestly upon acting on an arbitrary function in $C(P)$ and comparing the terms on both sides (taking into account that $s_j\lambda=\lambda -\langle\lambda,\alpha_j^\vee\rangle \alpha_j$).
\end{proof}

\begin{remark}\label{duality:rem}
By Eqs. \eqref{group-algebra-rel}, \eqref{cross-rel1}, the Demazure-Lusztig operators $\check{T}_j$ ($j=1,\ldots ,n$) together with the multiplicative action of the basis elements $e^\lambda$ ($\lambda\in P$) in fact determine
a representation $h\to \check{T}(h)$ ($h\in \mathcal{H}$) of the affine Hecke algebra on $\mathbb{C}[P]$  (extending the assignment $T_j\to\check{T}_j$ ($j=1,\ldots ,n$), $Y^\lambda \to e^\lambda$ ($\lambda\in P$)). Furthermore, the mapping
$T_w Y^\lambda \to Y^{\lambda} T_{w^{-1}}$ ($w\in W_0$, $\lambda\in P$) extends to an antilinear anti-involution
$\star$ of
$\mathcal{H}$ (agreeing with the previous $*$-anti-involution on the subalgebra $\mathcal{H}_0$).  With respect to the new $\star$-anti-involution and the pairing in Eq. \eqref{pairing} the integral-reflection representation in Proposition \ref{int-ref-rep:prp} is dual to the polynomial representation $\check{T}(\mathcal{H})$ in the sense that $( I(h) f , p)=(f,\check{T}(h^\star )p)$ ($h\in\mathcal{H}$, $f\in C(P)$, $p\in\mathbb{C}[P]$).
\end{remark}

\begin{remark}\label{int-refl:rem}
The integral-reflection operators $I_j$ \eqref{int-op1}, \eqref{int-op2} constitute a discrete counterpart of integral-reflection operators introduced by Gutkin and Sutherland in the context of their study of the spectral problem for the Laplacian perturbed by a delta potential supported on the reflection hyperplanes of the root system $R$
\cite{gut-sut:completely,gut:integrable}. Proposition \ref{int-ref-rep:prp} is the corresponding analog of
the observation in Ref. \cite{hec-opd:yang} that these Gutkin-Sutherland integral-reflection operators determine a representation of the Drinfeld-Lusztig graded affine Hecke algebra.
\end{remark}

\section{Diagonalization of $\hat{T}(\mathcal{Z})$}\label{sec5}
In this section we diagonalize the action of the center of $\mathcal{H}$ under our difference-reflection representation by means of Macdonald's spherical functions. Our main tool is an intertwining operator relating the difference-reflection representation to the auxiliary integral-reflection (or polynomial) representation.

We will employ the shorthand notation
$\hat{T}_w:=\hat{T}(T_w)$ and $I_w:=I(T_w)$ ($w\in W$).

\subsection{Intertwining operator}
Let
$\mathcal{J}:C(P)\to C(P)$ be the operator defined by
\begin{equation}\label{i-op}
(\mathcal{J}f)(\lambda ) := q_{\bar{u}_{\lambda}} (I_{\bar{u}_{\lambda} }^{-1}  f, 1) \qquad (f\in C(P),\lambda\in P),
\end{equation}
where $\bar{u}_\lambda:=w_o u_{w_o\lambda}w_o=t_\lambda w^{-1}_{\lambda}$  with
\begin{equation*}
w_\lambda :=w_o v_{w_o\lambda}w_o=v_{-\lambda}
\end{equation*}
(i.e. $w_\lambda$ is the shortest element of $W_0$ mapping $\lambda$ into the dominant cone $P^+$) and $(\cdot ,\cdot )$ refers to the pairing in Eq. \eqref{pairing}. Notice that
\begin{equation}
(\mathcal{J}f)(\lambda )=q_{t_\lambda}q_{w_\lambda}(I_{w_\lambda^{-1}}^{-1}f)(\lambda_+ )\quad \text{with}\quad \lambda_+:= w_\lambda \lambda .
\end{equation}
So in particular,
on the dominant cone $\mathcal{J}$ acts simply as a multiplication operator:  $(\mathcal{J}f)(\lambda )=q_{t_\lambda}f(\lambda)$ for  $\lambda\in P^+$.

\begin{theorem}[Intertwining Property]\label{intertwining:thm}
The operator $\mathcal{J}:C(P)\to C(P)$ \eqref{i-op} enjoys the following intertwining property, connecting the difference-reflection representation $ \hat{T}(\mathcal{H})$ with the integral-reflection representation
$ I(\mathcal{H})$
\begin{equation}\label{intertwining}
\hat{T}_w  \mathcal{J} = \mathcal{J} I_w,
\end{equation}
(for all $ w\in W$).
\end{theorem}

For $w\in W_0$, the intertwining property in Eq. \eqref{intertwining} is an immediate consequence of the next lemma, whose proof boils down to some straightforward computations based on  the well-known elementary Hecke algebra relations (cf. e.g. \cite[(4.1.2)]{mac:affine})
\begin{subequations}
\begin{eqnarray}
T_j T_w \!\!\!\!  &= &\!\!\! \! T_{s_jw}+\chi (w^{-1}\alpha_j)(q_j-q_j^{-1})T_w, \label{Hrela} \\
 T_w^{-1} T_j^{-1} \!\!\!\!  &= & \!\!\!\! T_{s_jw}^{-1}-\chi (w^{-1}\alpha_j)(q_j-q_j^{-1})T_w^{-1}, \label{Hrelb}
\end{eqnarray}
\end{subequations}
$j=1,\ldots ,n$ (where $\chi$ is in accordance with Eq. \eqref{length} and the second relation follows from the first one by applying the anti-involution $T_w\to T_w^{-1}$, $q\to q^{-1}$) together with
the observation that for $w,w^\prime\in W_0$ and $\lambda$ dominant
\begin{equation}\label{stable}
q_w( I^{-1}_{w}f)(\lambda ) = q_{w^\prime}( I^{-1}_{w^\prime}f)(\lambda) \quad\text{if}\quad
w^{-1}w^\prime\in W_{0,\lambda}
\end{equation}
(which is readily seen by induction on $\ell (w^{-1}w^\prime)$).

\begin{lemma}\label{finite-intertwining:lem}
The representations $\hat{T}(\mathcal{H}_0)$ and $I(\mathcal{H}_0)$ satisfy the  finite intertwining relations
\begin{equation}
\hat{T}_j  \mathcal{J} = \mathcal{J} I_j \qquad (j=1,\ldots ,n).
\end{equation}
\end{lemma}
\begin{proof}
Let $f \in C(P)$.
Elementary manipulations reveal that
\begin{eqnarray*}
\lefteqn{q_{t_\lambda}^{-1} (\hat{T}_j^{-1}\mathcal{J}f)(\lambda )} && \\
\!\!\!\! &\stackrel{\eqref{stable}}{=}& \!\!\! \! q_j^{-1} q_{w_\lambda} (I_{w_\lambda^{-1}}^{-1} f)( \lambda_+ )
+q_j^{-\text{sign} (w_\lambda \alpha_j)}
\bigl(
q_{w_\lambda s_j} (I_{s_j w_\lambda^{-1}}^{-1} f)( \lambda_+ ) -
q_{w_\lambda} (I_{w_\lambda^{-1}}^{-1} f)( \lambda_+ )
\bigr)
\\
\!\!\!\! &=& \!\!\! \!
 q_{w_\lambda}
 \left( (I_{s_j w_\lambda^{-1}}^{-1} f)( \lambda_+ )
 -\chi (w_\lambda \alpha_j )(q_j-q_j^{-1})
 ( I^{-1}_{w_\lambda^{-1}} f)( \lambda_+ )  \right)  \\
\!\!\!\! &\stackrel{\eqref{Hrelb}}{=}& \!\!\! \!  q_{w_\lambda}  ( I^{-1}_{w_\lambda^{-1}} I^{-1}_j f)( \lambda_+ ) =
q_{t_\lambda}^{-1} (\mathcal{J} I^{-1}_j f)(\lambda ),
\end{eqnarray*}
whence $\hat{T}_j^{-1}  \mathcal{J} = \mathcal{J} I_j ^{-1}$.
\end{proof}

The extension of the intertwining property in Eq. \eqref{intertwining} from $ W_0$ to $ W$ hinges on a second lemma, whose proof in contrast is technically somewhat more involved and therefore being relegated to Appendix \ref{appB}.

\begin{lemma}\label{affine-intertwining:lem}
The representation $\hat{T}(\mathcal{H})$ and $I(\mathcal{H})$ satisfy the affine intertwining relations
\begin{subequations}
\begin{eqnarray}
 \hat{T}_0  \mathcal{J}\!\!\! &= &\!\!\!\mathcal{J} I_0 \qquad (I_0:=I_{s_0} ), \\
    u \mathcal{J}\!\!\! &=& \!\!\!\mathcal{J} I_u\qquad (u\in \Omega ) .
\end{eqnarray}
\end{subequations}
\end{lemma}

\begin{remark}
By the duality in Remark \ref{duality:rem}, the action of the intertwining operator  can be rewritten in terms of
the polynomial representation $\check{T}(\mathcal{H})$ as
\begin{equation}
(\mathcal{J}f)(\lambda ) = q_{\bar{u}_{\lambda}} ( f, (\check{T}^{\star}_{\bar{u}_{\lambda} })^{-1} 1) \qquad (f\in C(P),\lambda\in P) ,
\end{equation}
where we have used the short-hand notation $\check{T}_w^\star :=\check{T}(T_w^\star)$.
\end{remark}

\begin{remark}\label{intertwining:rem}
The intertwining operator $\mathcal{J}$ \eqref{i-op} is a discrete counterpart
of the Gutkin-Sutherland propagation operator, which relates the spectral
problem for the Laplacian with a delta potential in Remark \ref{int-refl:rem} to that of the free Laplacian
\cite{gut-sut:completely,gut:integrable}.
From this perspective, Theorem \ref{intertwining:thm} yields the corresponding generalization of the fact that the propagation operator
in question intertwines
the integral-reflection representation and the Dunkl-type differential-reflection representation of the Drinfeld-Lusztig graded affine Hecke algebra of Refs. \cite{hec-opd:yang} and \cite{ems-opd-sto:periodic}, respectively. In fact, our difference-reflection
representation $\hat{T}(\mathcal{H})$, which was obtained by pushing the integral-reflection representation
$I(\mathcal{H})$ through the intertwining operator $\mathcal{J}$,
provides us with the Dunkl-type difference-reflection operators
for a discretization of the Laplacian with a delta potential on root hyperplanes  that was
introduced and studied in Ref. \cite{die:plancherel} (cf. also Remark \ref{discrete-laplacian:rem} below).
\end{remark}

\vspace{1ex}

\subsection{Bijectivity of the intertwining operator}
We will now show
that the intertwining operator $\mathcal{J}:C(P)\to C(P)$ is bijective.
The existence of this bijection intertwining the
difference-reflection representation $\hat{T}(\mathcal{H})$ and the integral-reflection representation
$I(\mathcal{H})$ reveals that these two representations of the affine Hecke algebra in $C(P)$ are in fact equivalent.
Moreover, it provides an alternative (indirect) proof of Theorem \ref{dif-ref-rep:thm} as a consequence
of Proposition \ref{int-ref-rep:prp}. Indeed, Lemmas \ref{finite-intertwining:lem}  and
\ref{affine-intertwining:lem}--- together with the bijectivity of $\mathcal{J}$---disclose
that the affine Hecke-algebra relations in Eqs. \eqref{qr}--\eqref{cr} may be seen as a consequence of
the corresponding relations  for $I_0,\ldots,I_n$ and $I_u$, $u\in\Omega$ (which follow in turn from
Proposition \ref{int-ref-rep:prp}).

To prove now the bijectivity in question some further notation is needed.
Let $\preceq$ represent the {\em dominance order} on the cone of dominant weights $P^+$ and
let $\leqslant$ denote the {\em Bruhat order} on the finite Weyl group $W_0$ \cite{bou:groupes,mac:affine}. Specifically,
$$\forall \lambda,\mu\in P^+:\quad  \mu \preceq \lambda \quad \text{iff} \quad \lambda-\mu\in Q^+$$ with
$Q^+:=\text{Span}_{\mathbb{Z}_{\geq 0}}(R^+)$,
and $\forall v,v^\prime \in W_0$: $v^\prime\leqslant v$ iff
$v^\prime=s_{i_1}\cdots s_{i_{p}}$ for a certain subsequence $(i_1,\ldots ,i_{p})$ of
$(j_1,\ldots ,j_\ell)$ with
$v=s_{j_1}\cdots s_{j_\ell}$ a reduced expression (i.e. $\ell=\ell (v)$).
The dominance order can be conveniently extended from $P^+$ to $P$ with the aid of the Bruhat order
(cf. Ref. \cite[Sec. 2.1]{mac:affine})
$$\forall \lambda ,\mu\in P:\quad  \mu\preceq\lambda\quad \text{iff}\quad \begin{cases} \mu_+\prec\lambda_+& (i),\\ \text{or} & \\
\mu_+=\lambda_+\ \text{and}\ w_\mu \leqslant w_\lambda & (ii).\end{cases}$$

\begin{theorem}[Automorphism]\label{bijectivity:thm}
The operator $\mathcal{J}$ \eqref{i-op} constitutes a linear automorphism of the space $C(P)$.
\end{theorem}

\begin{corollary}[Equivalence]
The difference-reflection representation $\hat{T}(\mathcal{H})$
and the integral-reflection representation $I(\mathcal{H})$ of the affine Hecke algebra
in $C(P)$ are equivalent:
\begin{equation*}
\hat{T}(h)= \mathcal{J} I(h) \mathcal{J}^{-1}\qquad \forall h\in \mathcal{H}.
\end{equation*}
\end{corollary}

\begin{proof}
It is clear that
the intertwining property in Theorem \ref{intertwining:thm}  and the invertibility of $\mathcal{J}$ ensure that
$\hat{T}(\mathcal{H})$ and $I(\mathcal{H})$ are equivalent representations of the affine Hecke algebra
in $C(P)$, i.e.  the corollary is in effect a direct consequence of the theorem.
The proof of the theorem---which amounts to showing that the linear operator $\mathcal{J}:C(P)\to C(P)$ is
bijective---is in turn immediate from the following triangularity property:
\begin{equation}\label{tria}
(I_{w_\lambda^{-1}}^{-1}f)(\lambda_+)=q^{-1}_{w_\lambda}f(\lambda )+\sum_{\mu\in P,\, \mu \prec \lambda} * f(\mu )\qquad (f\in C(P),\, \lambda\in P).
\end{equation}
Here and below the star symbols
$*$ refer to the expansion coefficients of lower terms (with respect to the partial order $\preceq$)
whose precise values are not relevant for the argument of the proof. Indeed, it is clear from the triangularity in Eq. \eqref{tria} that for any $g\in C(P)$
the linear equation $(\mathcal{J}f)(\lambda )=g(\lambda )$ ($\lambda\in P$) can be uniquely solved
inductively in $\lambda$ with respect to the partial order $\preceq$.

The triangularity in Eq. \eqref{tria} hinges on well-known saturation properties of the convex hull of the orbit of a weight with respect to the action of the finite Weyl group \cite{bou:groupes,mac:affine}.
For our purposes it is enough to recall that for any $\lambda\in P$ the weights in the convex hull of $W_0\lambda$ are given by the saturated set $P(\lambda) :=\{\mu\in P \mid  \mu_+\preceq \lambda_+\}$.
For $\lambda,\mu\in P$ one has that: (i) if $\mu \preceq \lambda$ then $\mu\in P(\lambda)$, and (ii)
if $\mu\in P(\lambda)$ then $[\mu, s_\alpha\mu] \subset P(\lambda)$ for any $\alpha\in R$, where  $s_\alpha:=s_{\alpha ,0}$ and
$[\mu ,s_{\alpha }\mu ]$ refers to the $\alpha$-string from $\mu$ to $s_\alpha\mu$, i.e.
$[\mu ,s_{\alpha }\mu ]:=\{ \mu-k\alpha\mid  k=0,\ldots ,\langle \mu,\alpha^\vee\rangle\}$.

After these preliminaries we are now in a position to
prove the triangularity in question by induction on $\ell (w_\lambda)$ starting from the straightforward case
that $\ell (w_\lambda)\leq 1$. (The case $\ell (w_\lambda)=0$ is in fact trivial since then $\lambda\in P^+$ and $(I_{w_\lambda^{-1}}^{-1}f)(\lambda_+)=f(\lambda) $).)
It is manifest from the explicit formula for the
action of $I_j$ (cf. Eqs. \eqref{int-op1}, \eqref{int-op2}) and the above properties of the saturated set $P(\lambda)$
that for $j=1,\ldots ,n$:
\begin{equation}\label{Ij}
(I_{j}^{-1}f)(\lambda ) =
\begin{cases}\displaystyle
q_j^{-1} f(s_j \lambda)+\sum_{\mu\in P,\, \mu \prec s_j \lambda} *f(\mu)
&\text{if}\ s_j \lambda \succ \lambda ,\\
\displaystyle\qquad\qquad\quad\
\sum_{\mu\in P,\, \mu \preceq s_j\lambda } *f(\mu )
&\text{if}\ s_j\lambda \preceq \lambda ,
\end{cases}
\end{equation}
which implies Eq. \eqref{tria} for $\lambda\in P$ such that $\ell (w_\lambda)=1$.
Upon picking $\lambda\in P$ such that the triangularity in Eq. \eqref{tria} holds for all $\mu\in P$ with
$\ell (w_\mu)\leq \ell (w_\lambda )$, it is readily seen that for any
$j\in \{1,\ldots ,n\}$ such that $\lambda \prec s_j\lambda$ (or equivalently
$\ell (w_{s_j\lambda})=\ell (w_\lambda )+1$) one has that
\begin{align}\label{Iw+1}
(I_{w_{s_j\lambda}^{-1}}^{-1}f)((s_j\lambda)_+)&\stackrel{(i)}{=}(I_{w_\lambda^{-1}}^{-1}I_{j}^{-1}f)(\lambda_+) ,\\
&\stackrel{(ii)}{=}
q_{w_\lambda}^{-1} (I_{j}^{-1} f)(\lambda)+\sum_{\nu \in P,\, \nu \prec \lambda} * (I_{j}^{-1}f)(\nu )  , \nonumber \\
&\stackrel{(iii)}{=}
q_{w_{s_j\lambda}}^{-1} f(s_j\lambda)+\sum_{\mu \in P,\, \mu \prec s_j\lambda} * f (\mu )
\nonumber
\end{align}
(thus completing the induction). Here Step $(i)$ of the derivation exploits that $w_{s_j\lambda}=w_\lambda s_j$ (with $\ell (w_\lambda s_j)=\ell (w_\lambda )+1$) and Step $(ii)$ relies on invoking of the induction hypothesis that
the triangularity holds for $w_\lambda$.
Step $(iii)$ follows in turn upon applying Eq. \eqref{Ij} to all terms  on the second line of Eq. \eqref{Iw+1}.
Indeed, if $\nu\prec\lambda\prec s_j\lambda$ then $s_j\nu\prec s_j\lambda $ (which is immediate from the definitions if $\nu_+ \prec \lambda_+$ and which follows from the elementary estimates
$w_{s_j\nu}  \leqslant w_\nu s_j  < w_\lambda s_j=w_{s_j\lambda}$ if $\nu_+=\lambda_+$ and
$w_\nu  < w_\lambda$).
\end{proof}

\subsection{Macdonald spherical functions}
Let $\xi\in V$. By $e^{i\xi}\in C(P)$ we denote the plane wave
 $e^{i\xi}(\lambda):=e^{i\langle \lambda ,\xi\rangle}=(e^{i\xi},e^{-\lambda}  )=(e^\lambda,e^{i\xi})$, $\lambda\in P$.
 By definition, the {\em Macdonald spherical function} $\Phi_\xi$, $\xi \in V$ is the function in $C(P)$
of the form
\begin{subequations}
\begin{equation}\label{mac-spha}
\Phi_\xi :=\mathcal{J}\phi_\xi\quad\text{with}\quad \phi_\xi :=  I(\mathbf{1}_0)e^{i\xi} ,
\end{equation}
where
\begin{equation}\label{mac-sphb}
\mathbf{1}_0:=\sum_{w\in W_0} q_w T_w .
\end{equation}
\end{subequations}
The Macdonald spherical function is $W_0$-invariant in the sense that
\begin{align*}
\Phi_\xi\in C(P)^{W_0}&:=\{ f\in C(P)\mid wf=f,\ w\in W_0\} \\
&=\{ f\in C(P)\mid \hat{T}_wf=q_w f,\ w\in W_0\} .
\end{align*}
 Indeed, since
$T_j\mathbf{1}_0=q_j\mathbf{1}_0$ for $1\leq j\leq n$ in view of Eq. \eqref{Hrela}, it is clear that
$\hat{T}_j\Phi_\xi= \hat{T}_j \mathcal{J} \phi_\xi= \mathcal{J} I_j \phi_\xi =  \mathcal{J} q_j \phi_\xi=q_j\Phi_\xi$, $j=1,\ldots, n$.

The symmetric monomials $m_\lambda:=\sum_{\mu\in W_0\lambda} e^\mu$, $\lambda\in P^+$ form a basis
of $\mathbb{C}[P]^{W_0}$.
For $p=\sum_{\lambda\in P^+} c_\lambda m_\lambda \in\mathbb{C}[P]^{W_0}$ ($c_\lambda\in\mathbb{C}$),
we define
$\widehat{p(Y)}:=\hat{T}(p(Y))$ where $p(Y):=\sum_{\lambda\in P^+} c_\lambda m_\lambda(Y)$.
The center of the affine Hecke algebra is then given by
$\mathcal{Z}=\{ p(Y) \mid p\in\mathbb{C}[P]^{W_0} \}$ and moreover
$\hat{T}(\mathcal{Z})=\{ \widehat{p(Y)} \mid p\in\mathbb{C}[P]^{W_0} \} $. Clearly the space $C(P)^{W_0}$ is stable under the action of $\hat{T}(\mathcal{Z})$.

\begin{theorem}[Diagonalization]\label{diagonal:thm}
The commuting subalgebra $\hat{T}(\mathcal{Z})\subset\hat{T}(\mathcal{H})$ is diagonalized by the Macdonald spherical function:
\begin{equation}
\widehat{p(Y)}\Phi_\xi = E_p(\xi) \Phi_\xi \quad\text{with}\quad E_p(\xi )=(p,e^{-i\xi}),
\end{equation}
for $\xi\in V$ and $p\in\mathbb{C}[P]^{W_0} $.
\end{theorem}

The proof of this theorem hinges on the intertwining operator and an explicit formula for
$\phi_\xi = I(\mathbf{1}_0)e^{i\xi}$ following from the work of Macdonald \cite{mac:spherical1,mac:spherical2}.

\begin{proposition}\label{mac-for:prp}
The function $\phi_\xi$ \eqref{mac-spha}, \eqref{mac-sphb} is given explicitly by
\begin{subequations}
\begin{equation}
\phi_\xi (\lambda ) = ( e^{i\xi}, P_{\lambda}) , \qquad \lambda\in P ,
\end{equation}
where
\begin{equation}\label{mac-form}
P_\lambda :=  \sum_{w\in W_0} e^{-w\lambda }
 \prod_{\alpha\in R^+} \frac{1-q_\alpha^2 e^{w\alpha }}{1-e^{w\alpha}}
\end{equation}
\end{subequations}
and $q_{\alpha}:=q_{\alpha ,0}$.
\end{proposition}

\begin{proof}
Let us recall from Remark \ref{duality:rem} that $( I(h)f,p)=(f,\check{T}({h^*})p)$ for $h\in\mathcal{H}_0$,  $f\in C(P)$,
$p\in\mathbb{C}[P]$, where $\check{T}(\mathcal{H}_0)$ refers to the standard polynomial representation of the finite Hecke algebra generated by the Demazure-Lusztig operators in Eq.  \eqref{demazure}.  The lemma is now an immediate consequence of Macdonald's celebrated formula $\mathbf{1}_0Y^{-\lambda} \mathbf{1}_0=P_\lambda (Y)\mathbf{1}_0$ ($\lambda\in P$) \cite[Thm. 1]{mac:spherical1} and (with more details)
\cite[(4.1.2)]{mac:spherical2} (see also e.g. \cite[Thm. 2.9(a)]{nel-ram:kostka} and \cite[Thm. 6.9]{par:buildings}). Indeed,  Macdonald's formula implies that $\check{T}(\mathbf{1}_0)e^{-\lambda}=P_\lambda$, whence
$\phi_\xi (\lambda)=( \phi_\xi , e^{-\lambda} ) = ( I(\mathbf{1}_0)e^{i\xi},e^{-\lambda})=(e^{i\xi}, \check{T}(\mathbf{1}_0)e^{-\lambda})=( e^{i\xi}, P_{\lambda})$.
\end{proof}

Proposition \ref{mac-for:prp} reveals that $\phi_\xi$ decomposes as a linear combination of plane waves $e^{iw\xi}$, $w\in W_0$ (with coefficients $\prod_{\alpha\in R^+} \frac{1-q_\alpha^2 e^{-i\langle w\xi ,\alpha\rangle }}{1-e^{-i\langle w\xi ,\alpha\rangle}} $).
With this information the proof of Theorem \ref{diagonal:thm} reduces to an elementary computation:
\begin{equation*}
\widehat{p(Y)}\Phi_\xi = \widehat{p(Y)}\mathcal{J} \phi_\xi =\mathcal{J} I (p(Y)) \phi_\xi=
\mathcal{J} p \phi_\xi = \mathcal{J} (p,e^{-i\xi}) \phi_\xi= (p,e^{-i\xi}) \Phi_\xi ,
\end{equation*}
where we have used that
$pe^{iw\xi}=(p,e^{-i\xi})e^{iw\xi} $ for $w\in W_0$, since $p\in\mathbb{C}[P]^{W_0}$
and
$e^\lambda e^{i\xi}=t_\lambda e^{i\xi}=e^{-i\langle\lambda ,\xi\rangle}e^{i\xi}=(e^\lambda,e^{-i\xi})e^{i\xi}$.

\begin{remark}\label{Macsph:rem}
It is immediate from Proposition \ref{mac-for:prp} and the $W_0$-invariance of the Macdonald spherical function
$\Phi_\xi$ that
\begin{equation}
\Phi_\xi (\lambda )=
q_{t_{\lambda}}  \sum_{w\in W_0} e^{i\langle{w\xi ,\lambda_+ \rangle} }
 \prod_{\alpha\in R^+} \frac{1-q_\alpha^2 e^{-i\langle w\xi , \alpha\rangle }}{1-e^{-i\langle w\xi , \alpha\rangle}} ,
 \qquad \lambda\in P .
\end{equation}
\end{remark}

\section{Unitarity}\label{sec6}
In this section we describe a Hilbert space structure for which our difference-reflection representation becomes unitary.

Here it is always assumed that $q:W\to (0,1)$.
We will employ the shorthand notation $X(q^2):=\sum_{w\in X} q_w^2$ for $X\subset W_0$. So in particular, $W_0(q^2)$  and $W_{0,x}(q^2)$ ($x\in V$) represent the (generalized) Poincar\'e series of $W_0$ and $W_{0,x}$ associated with $q^2$, respectively.
Let $l^2(P,\delta )$ be the Hilbert space of functions $\{ f\in C(P)\mid \langle f,f\rangle_\delta <\infty \}$, where\begin{subequations}
\begin{equation}
\langle f ,g\rangle_\delta :=\sum_{\lambda\in P}  f(\lambda) \overline{g(\lambda)} \delta_\lambda \qquad
(f,g\in l^2(P,\delta )),
\end{equation}
with
\begin{equation}\label{delta}
\delta_\lambda := \mathcal{N}_0^{-1}\, q^{-2}_{u_\lambda} =
 \mathcal{N}_0^{-1} \prod_{\substack{\alpha\in R^+,k\in\mathbb{Z}\\ V_{\alpha ,k}\in S(\lambda)}}   q^{-2}_{\alpha ,k} ,
 \qquad \mathcal{N}_0:= W_0(q^2)
\end{equation}
\end{subequations}
and $S(\lambda):=S(u_\lambda )=\{ V_{\alpha ,k}\mid V_{\alpha ,k}\; \text{separates}\; \lambda\;\text{and}\; A\}$ (cf. \cite[(2.4.8)]{mac:affine}).

\begin{theorem}[Unitarity of $\hat{T}(\mathcal{H})$]\label{unitary1:thm}
The difference-reflection representation $h\to\hat{T}(h)$ ($h\in\mathcal{H}$) on $C(P)$ restricts to a unitary representation of the affine Hecke algebra
into the space of bounded operators on $l^2(P,\delta )$, i.e.
\begin{equation}
\langle \hat{T}(h)f,g\rangle_\delta =
\langle f,\hat{T}(h^*)g\rangle_\delta \qquad (h\in\mathcal{H},\ f,g\in l^2(P,\delta )).
\end{equation}
\end{theorem}
\begin{proof}
Let $f,g\in l^2(P,\delta )$. It suffices to show that the actions of $\hat{T}_j$ ($0\leq j\leq n$) and  $u$ ($u\in\Omega$) determine bounded operators on $l^2(P,\delta )$ satisfying
(i) $\langle \hat{T}_j f,g\rangle_\delta=\langle f, \hat{T}_j g\rangle_\delta$ and  (ii) $\langle u f,g\rangle_\delta=\langle f, u^{-1} g\rangle_\delta$.
Property (ii)  follows by performing the change of coordinates $\lambda\to u\lambda$ to the (discrete) integral $\langle u f,g\rangle_\delta$. Indeed, invoking of the symmetry $\delta_{u\lambda}=\delta_\lambda$ (as
$S(u\lambda )=S(\lambda )$) then produces the integral $\langle  f,u^{-1}g\rangle_\delta$.
Property (i) follows in turn by performing the change of coordinates $\lambda\to s_j\lambda$ to the integral $\langle \chi_j s_j f,g\rangle_\delta$, which entails the integral  $\langle f,  \chi_j s_j g\rangle_\delta$.
Here one uses the symmetries
$s_j\chi_j=\chi_j^{-1}s_j$ and
$\delta_{s_j\lambda}=\chi_j^2(\lambda )\delta_\lambda$ (as
 $S(s_j\lambda )=S(\lambda )\setminus \{ V_j\}$ if $V_j\in S(\lambda )$,
$S(s_j\lambda )=S(\lambda)$ if $\lambda\in V_j$, and
$S(s_j\lambda )=S(\lambda )\cup \{ V_j\}$  otherwise). The computations in question also reveal that the actions of
$u$  and $s_j$ (and thus that of $\hat{T}_j$) are indeed bounded in  $l^2(P,\delta )$
(as $\langle uf ,uf\rangle_\delta =\langle f ,f\rangle_\delta$
and $\langle s_jf ,s_j f\rangle_\delta =\langle\chi_js_j f ,\chi_j^{-1}s_jf\rangle_\delta=\langle f ,\chi_j s_j\chi_j^{-1}s_jf\rangle_\delta=\langle f ,\chi_j^{2}f\rangle_\delta$, and
$\chi_j$ is a bounded function on $P$).
\end{proof}
Since $P^+$ is a fundamental domain for the action of $W_0$ on $P$, the symmetric subspace
$l^2(P,\delta )^{W_0}:=l^2(P,\delta )\cap C(P)^{W_0}$ can be identified with the Hilbert space $l^2(P^+,\Delta )$ of
functions $\{ f:P^+\to \mathbb{C} \mid \ \langle f,f\rangle_\Delta <\infty \}$, where
\begin{subequations}
\begin{equation}
\langle f ,g\rangle_\Delta :=\sum_{\lambda\in P^+}  f(\lambda) \overline{g(\lambda)} \Delta_\lambda
\qquad (f,g\in l^2(P^+,\Delta )) ,
\end{equation}
with
\begin{equation}\label{mac-mes}
\Delta_\lambda :=  \sum_{\mu\in W_0\lambda} \delta_\mu =
q^{-2}_{t_\lambda} \frac{W^\lambda_0(q^2)}{W_0(q^2)}=  \frac{q^{-2}_{t_\lambda}}{W_{0,\lambda}(q^2)}
 \qquad (\lambda\in P^+)
\end{equation}
\end{subequations}
and $W_0^\lambda := \{ w_\mu \mid \mu \in W_0\lambda \}$.
The first equality in Eq. \eqref{mac-mes}
follows from the relations
$q_{u_\mu}=q_{t_\mu}q^{-1}_{v_\mu}=q_{t_{\mu_+}}q^{-1}_{v_\mu}$ and $q_{v_\mu}=q_{w_{w_o\mu}} $ ($\mu \in P$); the second equality is readily inferred upon observing that the mapping $(w,w^\prime)\to ww^\prime$ determines a bijection of
$ W_{0,\lambda}\times W_0^\lambda$ onto $W_0$ satisfying $\ell (ww^\prime)=\ell (w)+\ell (w^\prime )$,
whence $q_{ww^\prime}=q_wq_{w^\prime}$ and thus $W_0(q^2)=W_{0,\lambda}(q^2)W_0^\lambda (q^2)$.

The following adjointess relations for the basis elements $\widehat{m_\lambda(Y)}$ are an immediate consequence of the unitarity in Theorem \ref{unitary1:thm} (recall in this connection also the last paragraph of Section \ref{sec2}).
\begin{corollary}[Adjointness Relations in $\hat{T}(\mathcal{Z})$]\label{sadjoint:cor}
The basis operators $ \widehat{m_\lambda (Y)}$, $\lambda\in P^+$ spanning $\hat{T}(\mathcal{Z})$ satisfy the adjointness relations
\begin{equation}
\langle \widehat{m_\lambda (Y)}f,g\rangle_\Delta =
\langle f, \widehat{m_{\lambda^*} (Y)}g\rangle_\Delta \qquad (\lambda\in P^+,\  f,g\in l^2 (P^+,\Delta )) .
\end{equation}
\end{corollary}
In particular, it is evident from Corollary \ref{sadjoint:cor} that the symmetrized operators $(\widehat{m_\lambda (Y)}+\widehat{m_{\lambda^*} (Y)})$ and $i(\widehat{m_\lambda (Y)}-\widehat{m_{\lambda^*} (Y)})$ are self-adjoint in
$l^2 (P^+,\Delta )$.

\vspace{1ex}
\begin{remark}\label{rem-expl} Let
$$e_q(\lambda):=\prod_{\alpha\in R^+}q_\alpha^{\langle \lambda,\alpha^\vee\rangle},\qquad
\lambda\in P.$$
It is instructive to recall to mind
that $q_{t_\lambda}$  and $W_{0,\lambda}(q^2)$
can be conveniently written explicitly in terms of the multiplicity function via the evaluation formula $q_{t_\lambda}=e_q(\lambda_+ )$ (by Eq. \eqref{delta} with the symmetries
$q_{t_\lambda}=q_{t_{w_o(\lambda_+)}}=q_{u_{w_o(\lambda_+)}}$, $q_{\alpha ,k}=q_\alpha$)
and Macdonald's classic product formula
\begin{equation}\label{poincare}
W_{0,\lambda}(q^2)=\prod_{\substack{\alpha\in R^+\\ \langle\lambda,\alpha^\vee\rangle =0}} \frac{1-q_\alpha^2 e_q(\alpha)}{1-e_q(\alpha )},
\end{equation}
respectively.
In particular, evaluation of the RHS of  $\Delta$ \eqref{mac-mes}
produces
\begin{equation}\label{mac-mes2}
\Delta_\lambda = e_q(-2\lambda ) \prod_{\substack{\alpha\in R^+\\ \langle\lambda,\alpha^\vee\rangle =0}} \frac{1- e_q(\alpha)}{1-q_\alpha^2 e_q(\alpha )} \qquad (\lambda\in P^+).
\end{equation}
\end{remark}

\begin{remark}
Let $\text{Vol}(A):=\int_A \text{d}\xi$, where $\text{d}\xi$ denotes the Lebesgue measure on $V$, and let
\begin{equation}
 \breve{\Delta}(\xi ):= \breve{\mathcal{N}}_0^{-1}\prod_{\alpha\in R^+}
 \left| \frac{1-e^{i\langle \alpha,\xi\rangle}}{1-q_\alpha^2 e^{i\langle \alpha,\xi\rangle}} \right|^2,
 \qquad \breve{\mathcal{N}}_0:=   (2\pi )^{n}|W_0| \text{Vol}(A) .
\end{equation}
For
$\breve{f},\breve{g}$ in the Hilbert space $L^2 (2\pi A,  \breve{\Delta}(\xi ) \text{d}\xi )$, their
inner product is written as
\begin{equation}
\langle \breve{f }, \breve{g}\rangle_{\breve{\Delta}} :=
\int_{2\pi A} \breve{f}(\xi ) \overline{\breve{g}(\xi )} \breve{\Delta}(\xi )\text{d} \xi  .
\end{equation}
It is well-known from Macdonald's theory \cite{mac:spherical1,mac:spherical2} (cf. also \cite[\S 10]{mac:orthogonal})
that the measure $\Delta$ \eqref{mac-mes2} turns the Fourier-Macdonald pairing
\begin{subequations}
\begin{equation}
 f=\mathcal{F}_q(\breve{f}):= \langle \breve{f} , \Phi (\cdot )\rangle_{\breve{\Delta}} =
 \int_{2\pi A} \breve{f}(\xi ) \overline{\Phi_\xi (\cdot )} \breve{\Delta}(\xi )\text{d} \xi ,
 \end{equation}
 with the kernel function (cf. Remark \ref{Macsph:rem})
 \begin{equation}
 \Phi_\xi (\lambda )=e_q(\lambda) (e^{i\xi}, P_{\lambda}) = e_q(\lambda)
 \sum_{w\in W_0} e^{i\langle{w\xi ,\lambda \rangle} }
 \prod_{\alpha\in R^+} \frac{1-q_\alpha^2 e^{-i\langle w\xi , \alpha\rangle }}{1-e^{-i\langle w\xi , \alpha\rangle}}
 \end{equation}
 ($\xi\in 2\pi A$, $\lambda\in P^+$), into an Hilbert space isomorphism
$\mathcal{F}_q: L^2 (2\pi A,  \breve{\Delta}(\xi ) \text{d}\xi )\to l^2(P^+,\Delta )$
with the inversion formula given by
\begin{equation}
\breve{f}=\mathcal{F}_q^{-1}(f)= \langle f , \Phi_{\cdot}\rangle_\Delta =\sum_{\lambda\in P^+} f (\lambda) \overline{\Phi_{\cdot} (\lambda)} \Delta_\lambda
\end{equation}
\end{subequations}
(where the dot $\cdot$ refers to the suppressed argument).
From this perspective, Theorem \ref{diagonal:thm}  (with $\xi\in 2\pi A$) provides the spectral decomposition
$\mathcal{F}_q \circ  E_p   \circ \mathcal{F}_q^{-1}$
of the bounded normal discrete difference operator $\widehat{p(Y)}$ in the Hilbert space $l^2(P^+,\Delta )\cong l^2(P,\delta )^{W_0}$ (where $E_p$ refers to the multiplication operator
$(E_p\breve{f})(\xi ):= E_p(\xi)\breve{f}(\xi )$ on $L^2 (2\pi A,   \breve{\Delta}(\xi )\text{d}\xi )$).
For $\omega$ a (quasi-)minuscule weight, the explicit action of the corresponding difference operator
$\widehat{m_\omega (Y)}$ in $l^2(P^+,\Delta )$ is provided by Corollary \ref{action-symmetric:cor} below.
\end{remark}

\section{The explicit action of $\widehat{m_\omega(Y)}$ and associated Pieri formulas}\label{sec7}
Throughout this section it will be assumed that $\omega\in P^+$ is (quasi-)minuscule (cf. Appendix \ref{appB} below).
By computing the action of $\widehat{m_\omega(Y)}$ on $C(P)$ in closed form, Theorem \ref{diagonal:thm} gives rise
to an explicit Pieri formula for the Macdonald spherical functions. To describe the action in question let us introduce a similarity transformation
$\epsilon :C(P)\to C(P)$
and a difference operator $M_\omega:C(P)\to C(P)$
of the form
$(\epsilon f)(\lambda):= q_{t_\lambda} f(w_o\lambda )$ ($f\in C(P)$, $\lambda\in P$) and
\begin{subequations}
\begin{equation}\label{Ma}
(M_\omega f)(\lambda):=
\sum_{\nu\in W_0\omega} \Bigl( a_{\lambda ,\nu} f(\lambda -\nu)+ b_{\lambda ,\nu}f(\lambda) \Bigr)\qquad
(f\in C(P),\lambda\in P),
\end{equation}
with
\begin{equation}
a_{\lambda ,\nu}:=q_{w_{w_\lambda (\lambda-\nu)}}
q_{w_{w_\lambda (\lambda-\nu)}w_\lambda}q_{w_\lambda}^{-1}\quad  \text{and}\quad
b_{\lambda,\nu}:= \varepsilon_{\lambda,\nu}(1-q_0^{-2})e_q(w_\lambda \nu),
\end{equation}
where $e_q(\cdot)$ is as defined in
Remark \ref{rem-expl},
\begin{equation}\label{Mb}
\varepsilon_{\lambda ,\nu} :=
\begin{cases}
\theta (w_\lambda (\lambda -\nu)) &\text{if}\ (\lambda-\nu)_+\neq\lambda_+ \\
\chi (\nu ) & \text{if}\ (\lambda-\nu)_+ =\lambda_+ \\
\end{cases} ,
\end{equation}
\end{subequations}
$\theta (\mu):= \langle \mu_+-\mu,\rho^\vee\rangle -\ell (w_\mu)$, and $\rho^\vee:=\frac{1}{2}\sum_{\alpha\in R^+}\alpha^\vee $.

\begin{theorem}\label{action:thm} One has that
$\widehat{m_\omega(Y)}=\epsilon M_\omega \epsilon^{-1}$.
\end{theorem}

\begin{corollary}\label{action-symmetric:cor}
The restriction of the action of $\widehat{m_\omega(Y)}$ to $C(P)^{W_0}\cong C(P^+)$ is given by
\begin{subequations}
\begin{eqnarray}\label{Msa}
\lefteqn{(\widehat{m_\omega(Y)} f)(\lambda)=} && \\
&&U_{\lambda ,-\omega}(q^2)f(\lambda ) +\sum_{\substack{  \nu\in W_0\omega \\ \lambda -\nu\in P^+}} V_{\lambda ,-\nu}(q^2)f(\lambda-\nu)\qquad
(f\in C(P^+),\lambda\in P^+), \nonumber
\end{eqnarray}
with
\begin{equation}\label{Msb}
V_{\lambda ,\nu }(q^2) := e_q(-\nu) \prod_{\substack{ \alpha\in R^+\\ \langle \lambda ,\alpha^\vee\rangle =0 \\
\langle \nu ,\alpha^\vee\rangle >0}}
\frac{1-q_\alpha^2 e_q(\alpha)}{1-e_q(\alpha)}
\end{equation}
and
\begin{eqnarray}\label{Msc}
\lefteqn{U_{\lambda ,\mu }(q^2):= } && \\
&& \begin{cases}\displaystyle
0 & \text{for}\ \mu_+ \ \text{minuscule}, \\
\displaystyle \sum_{\nu\in W_0 \mu} e_q(\nu) -\sum_{\substack{ \nu\in W_0\mu\\\lambda +\nu\in P^+}}
V_{\lambda,\nu}(q^2) & \text{for}\ \mu_+ \ \text{quasi-minuscule} .
\end{cases} \nonumber
\end{eqnarray}
\end{subequations}
\end{corollary}

The diagonalization Theorem \ref{diagonal:thm} combined with the symmetric reduction
(Corollary \ref{action-symmetric:cor})
of the explicit action of $\widehat{m_\omega(Y)}$ (Theorem \ref{action:thm}),
immediately produces the following Pieri formula expressing
the multiplicative action of $m_\omega$ in $\mathbb{C}[P]^{W_0}$ in terms of
the
Macdonald spherical basis
$p_\lambda:=e_q(\lambda)P_{\lambda^*}$, $\lambda\in P^+$.
\begin{corollary}[Pieri formula]\label{pieri:cor} One has that
\begin{equation}
m_\omega p_\lambda = U_{\lambda ,\omega}(q^2)p_\lambda +\sum_{\substack{  \nu\in W_0\omega \\ \lambda +\nu\in P^+}} V_{\lambda ,\nu}(q^2)p_{\lambda+\nu}\quad (\lambda\in P^+).
\end{equation}
\end{corollary}
The Pieri formula in Corollary \ref{pieri:cor} is a special case of Pieri formulas for the Macdonald spherical functions
obtained via degeneration descending from the level of  the Macdonald polynomials \cite{die-ems:pieri}. For root systems of type $A$ and $\omega$ minuscule the Pieri formula under consideration amounts to a classic Pieri formula for the Hall-Littlewood polynomials due to Morris \cite{mor:note} (cf. Appendix \ref{appC} below).

\begin{remark}
The factor $\varepsilon_{\lambda ,\nu}$ in the coefficients of  $M_\omega$ \eqref{Ma},\eqref{Mb}
takes values in $\{ 0,1\}$.
For $\omega$ minuscule the factor in question vanishes (so $b_{\lambda,\nu}=0$) and the coefficient
$a_{\lambda ,\nu}$ simplifies to $q_{w_{w_\lambda (\lambda-\nu)}}^2$.
\end{remark}

\begin{remark}
It is manifest from the relations in Corollary \ref{sadjoint:cor} that
the adjoint of
$\widehat{m_\omega(Y)}$ \eqref{Msa}--\eqref{Msc} in the Hilbert space
$\ell^2 (P^+,\Delta )$
is given by the action of $\widehat{m_{\omega^*}(Y)}$ on $\ell^2 (P^+,\Delta )$.
More generally, it follows from the unitarity in Theorem \ref{unitary1:thm} that
the adjoint of $\widehat{m_\omega(Y)}$ \eqref{Ma}, \eqref{Mb}
in the Hilbert space $\ell^2 (P,\delta )$
is
given by the action of $\widehat{m_{\omega^*}(Y)}$ on $\ell^2 (P,\delta )$.
Since
$\omega^*$ is (quasi-)minuscule if (and only if) $\omega$ is (quasi-)minuscule, this means that
these adjoints are given by the same formulas of Corollary \ref{action-symmetric:cor} and Theorem \ref{action:thm}, respectively, with $\omega$ being replaced by $\omega^*$.
In particular, for $\omega$ quasi-minuscule the operators in question are self-adjoint
(as in this situation $\omega^*=\omega$).
\end{remark}

\begin{remark}\label{discrete-laplacian:rem}
Corollary \ref{action-symmetric:cor} provides an explicit formula for the discretization of the Laplacian with delta potential associated with $R$ from
Ref. \cite{die:plancherel} (cf. Remark \ref{intertwining:rem}).
\end{remark}

\begin{remark}
The standard polynomial representation of the affine Hecke algebra in terms of Demazure-Lusztig operators (dual to our integral-reflection representation $I(\mathcal{H})$)
was extended by Cherednik to a representation of the double affine Hecke algebra  \cite{che:double,mac:affine}. The representation in question contains Dunkl-type $q$-difference-reflection operators that were used for the construction of Macdonald's commuting $q$-difference operators diagonalized by the Macdonald polynomials
\cite{che:double,mac:affine}. Since Macdonald's polynomials are a $q$-deformation of the Macdonald spherical functions \cite{mac:orthogonal},
our difference-reflection representation $\hat{T}(\mathcal{H})$ is expected to correspond to a suitable degeneration of
Cherednik's representation of the double affine Hecke algebra (therewith linking the latter representation to
the differential-reflection representation of the
graded affine Hecke algebra in Ref. \cite{ems-opd-sto:periodic}).
\end{remark}

\subsection{Proof of Theorem \ref{action:thm}}
Since $\ell (w_ow)=\ell (w_o)-\ell(w)$ for any $w\in W_0$, it follows that
$q_{w_ow^{-1}}=q_{w_o} q_w ^{-1}$ and
$T_{w_ow^{-1}}^{-1} T_{w_o}=T_{w}$, whence
\begin{equation}\label{J-P}
q_{w_ow^{-1}}(I^{-1}_{w_ow^{-1}} I_{w_o}f)(\lambda)=q_{w_o}q_w^{-1}(I_w f)(\lambda)\quad (f\in C(P),\lambda\in P, w\in W_0).
\end{equation}
Combined with Eq. \eqref{stable}, this yields the following stability property for $w,w^\prime\in W_0$ and $\lambda\in P^+$
\begin{equation}\label{stable2}
q_w^{-1}(I_w f)(\lambda ) =q_{w^\prime}^{-1}(I_{w^\prime} f)(\lambda )\quad\text{if}\quad w (w^\prime)^{-1}\in W_{0,\lambda}.
\end{equation}

Let us now abbreviate $I(m_\omega(Y))=\sum_{\nu\in W_0\omega } t_\nu$ as $m_\omega(t)$.
In view of the intertwining relations (Theorem \ref{intertwining:thm}) and the bijectivity
of the intertwining operator $\mathcal J$ (Theorem \ref{bijectivity:thm}),
it is sufficient for proving the theorem
to show that $$\mathcal J m_\omega(t)=\epsilon  M_\omega \epsilon ^{-1}\mathcal J,$$
or equivalently (since $m_\omega(Y)\in \mathcal{Z}(\mathcal{H})$), that
$$
\epsilon^{-1}\mathcal J I_{w_o} m_\omega(t)= M_\omega \epsilon ^{-1} \mathcal J I_{w_o}.
$$
Relation \eqref{J-P} and stability properties in Eqs. \eqref{stable}, \eqref{stable2} imply that
\begin{equation}\label{Pop}
(\epsilon ^{-1}\mathcal J  I_{w_o}f)(\lambda)= q_{w_o} q_{w_\lambda}^{-1}(I_{w_\lambda}f)(\lambda_+) \quad (f\in C(P), \lambda\in P),
\end{equation}
and thus (using again that $m_\omega(Y)\in \mathcal{Z}(\mathcal{H})$)
$$
(\epsilon^{-1}\mathcal J I_{w_o} m_\omega(t)f)(\lambda)=q_{w_o} q_{w_\lambda}^{-1}\sum_{\nu\in W_0\omega} (I_{w_\lambda} f)(w_\lambda(\lambda-\nu)).
$$
That this expression is equal to $ (M_\omega  \epsilon^{-1}\mathcal J I_{w_o} f)(\lambda)$  hinges on the identity
\begin{eqnarray}\label{qm2}
\lefteqn{a_{\lambda ,\nu}( \epsilon^{-1}\mathcal J I_{w_o}f)(\lambda-\nu)
=  } && \\
&& q_{w_o}q^{-1}_{w_\lambda} \bigl( (I_{w_\lambda} f)(w_\lambda(\lambda-\nu))
 -  \varepsilon_{\lambda,\nu} (1-q_0^{-2}) e_q(w_\lambda\nu) (I_{w_\lambda} f)(\lambda_+)\bigr) \nonumber
\end{eqnarray}
(combined with Eq. \eqref{Pop}).
To infer the identity in Eq. \eqref{qm2} the following lemmas are instrumental.

\begin{lemma}\label{tlem1}
For $\lambda\in P^+$ and $\nu\in W_0\omega$, we are in either one of the following two situations:
(i) if $(\lambda-\nu)_+\neq \lambda$ then
$w_{\lambda-\nu}\in W_{0, \lambda}$  and
$$\theta (\lambda-\nu) =
\begin{cases} 1&\text{for}\ \nu\in R(w_{\lambda-\nu})\\
0&\text{for}\ \nu\not\in R(w_{\lambda-\nu})
\end{cases},\quad \text{where}\ R(w):=R^+\cap w^{-1}(R^-),
$$
or
(ii) if $(\lambda-\nu)_+=\lambda$ then
$w_{\lambda-\nu}\nu=-\alpha_j$ for some $j\in \{ 1,\ldots ,n\}$, moreover,
$s_jw_{\lambda-\nu}\in W_{0, \lambda}$,  $\theta(\lambda-\nu)=0$,
$R(w_{\lambda -\nu})=R(s_jw_{\lambda -\nu})\cup\{\nu\} $ and $q_j=q_0$.
\end{lemma}

Before embarking on the proof of this lemma, let us first highlight some crucial (though elementary) observations.
For any $\lambda\in P$ the set $R(w_\lambda)$ is given by
$$R(w_\lambda)=\{ \alpha\in R^+ \mid \langle \lambda,\alpha^\vee \rangle < 0 \}$$ (cf.
\cite[\text{Eq.}~(2.4.4)]{mac:affine}) and  for any simple root $\alpha_j\in R(w_{\lambda})$ we have that
$w_\lambda s_j=w_{s_j\lambda}$ with $\ell (w_{s_j\lambda})=\ell( w_\lambda )-1$.
Given $\lambda\in P^+$, $\nu\in W_0\omega$ and $\mu=s_j(\lambda-\nu)$ with $\alpha_j\in R(w_{\lambda-\nu})$, we are in one of the following three cases:
\begin{itemize}
\item[$(A)$]$\langle\lambda,\alpha_j^\vee\rangle =0$ and $\langle\nu,\alpha_j^\vee\rangle=1$. Then $\mu=\lambda-s_j\nu\in \lambda-W_0\omega$ and  $\theta(\lambda-\nu)=\theta(\mu)$. (Notice that $s_j\in W_{0,\lambda}$.)
\item[$(B)$] $\langle \lambda,\alpha_j^\vee\rangle =0$ and $\langle \nu,\alpha_j^\vee\rangle =2$. Then $\mu=\lambda-s_j\nu=\lambda+\alpha_j\in \lambda-W_0\omega$ and $\theta(\lambda-\nu)=\theta(\mu)+1$. (Notice that $s_j\in W_{0,\lambda}$ and $\nu=\alpha_j$.)
\item[$(C)$] $\langle \lambda,\alpha_j^\vee\rangle =1$ and $\langle\nu,\alpha_j^\vee\rangle=2$.  Then
$\mu=\lambda$ and $\theta(\lambda-\nu)=\theta (\mu)=0$. (Notice that
$w_{\lambda-\nu}=s_j$ and $\nu=\alpha_j$.)
\end{itemize}
It is moreover evident that in the Cases $(B)$ and $(C)$, which occur only when $\omega$ is quasi-minuscule, one has that
$q_j=q_0$ (since $\alpha_j\in W_0\omega$ with $\omega=\alpha_0$).

\begin{proof}[Proof of Lemma \ref{tlem1}]
It is sufficient to restrict attention to the case that
$\lambda -\nu\not\in P^+$ (as for $\lambda-\nu\in P^+$ the lemma is trivial). For a reduced
decomposition
$w_{\lambda-\nu}=s_{j_\ell}\cdots s_{j_1}$ with $\ell =\ell (w_{\lambda-\nu})\geq 1$, we write
$$\nu_k:= s_{j_{k}}\cdots s_{j_1} \nu\quad  \text{for}\quad k=0,\ldots ,\ell$$ and
$$\beta_k:= s_{j_1}\cdots s_{j_{k}}\alpha_{j_{k+1}}\quad  \text{for}\quad k=0,\ldots ,\ell-1$$ (with the conventions that
$\nu_0:=\nu$ and $\beta_0:=\alpha_{j_1}$). This means that $$R(w_{\lambda-\nu})=\{ \beta_0,\ldots ,\beta_{\ell-1}\}$$ (cf. \cite[(2.2.9)]{mac:affine}).  It is immediate from the Observations $(A)$-$(C)$ above that the minimal sequence of weights taking $\lambda-\nu$ to $(\lambda-\nu)_+$ by successive application of the simple reflections
in our reduced decomposition of $w_{\lambda-\nu}$ is either of the form
(Situation $(i)$):
\begin{equation}\label{chain1}
\lambda-\nu=\lambda-\nu_0\stackrel{s_{j_1}}{\longrightarrow}\lambda-\nu_1\stackrel{s_{j_2}}{\longrightarrow}\cdots  \stackrel{s_{j_{\ell-1}}}{\longrightarrow}
\lambda-\nu_{\ell-1} \stackrel{s_{j_\ell}}{\longrightarrow} \lambda-\nu_\ell =(\lambda-\nu)_+ ,
\end{equation}
or of the form (Situation $(ii)$):
\begin{equation}\label{chain2}
\lambda-\nu=\lambda-\nu_0\stackrel{s_{j_1}}{\longrightarrow}\lambda-\nu_1\stackrel{s_{j_2}}{\longrightarrow}\cdots  \stackrel{s_{j_{\ell-1}}}{\longrightarrow}
\lambda-\nu_{\ell-1} \stackrel{s_{j_\ell}}{\longrightarrow} \lambda =(\lambda-\nu)_+ ,
\end{equation}
because Case $(C)$ can at most occur
at the last step:  $\lambda-\nu_{\ell-1} \stackrel{s_{j_\ell}}{\longrightarrow}  (\lambda-\nu)_+$ (as
this case takes us back to $P^+$).
In Situation $(i)$ (i.e. Case $(C)$ does not occur at the last step) we have that
$$w_{\lambda-\nu}\in W_{0,\lambda}\quad\text{and}\quad (\lambda-\nu)_+\neq \lambda ,$$ whereas in Situation $(ii) $ (i.e. Case $(C)$ does occur at the last step) we have that
 $$s_{j_\ell}w_{\lambda-\nu}=s_{j_{\ell-1}}\cdots s_{j_1}\in W_{0,\lambda},\quad q_{j_\ell}=q_0\quad\text{and}\quad
 (\lambda-\nu)_+=\lambda .$$ Moreover, in the latter situation
 $\nu_{\ell -1}=\alpha_{j_\ell}$, i.e.
 $w_{\lambda-\nu}\nu=-\alpha_{j_\ell}$ and
 $$\nu =(s_{j_{\ell}}w_{\lambda-\nu})^{-1}\alpha_{j_\ell}=s_{j_1}\cdots s_{j_{\ell-1}}\alpha_{j_\ell}=\beta_{\ell-1}\in
 R(w_{\lambda-\nu})\setminus R(s_{j_\ell}w_{\lambda-\nu}).$$ It remains to compute $\theta (\lambda-\nu )$. Since $\theta ((\lambda-\nu)_+)=0$, it is clear from the Observations $(A)$-$(C)$ that $\theta (\lambda-\nu )$ is equal to the number of
 times Case $(B)$ occurs in the above sequences, i.e. the number of times that
 $$\langle \nu_k,\alpha_{j_{k+1}}^\vee \rangle =2\quad \text{for}\quad k=0,\ldots ,\ell^\prime -1,$$ with $\ell^\prime =\ell$ in Situation $(i)$ and $\ell^\prime=\ell -1$ in Situation $(ii)$. Since for $k=0,\ldots ,\ell^\prime -1$: $$\langle \nu_k,\alpha_{j_{k+1}}^\vee \rangle =2\Leftrightarrow\langle \nu ,\beta_k^\vee \rangle =2\Leftrightarrow \nu=\beta_k,$$
 it is clear that in Situation $(i)$ $\theta (\lambda-\nu)$ is equal to $0$ or $1$ depending whether $\nu\not\in R(w_{\lambda-\nu})$ or $\nu\in R(w_{\lambda-\nu})$, respectively, and in
 Situation $(ii)$ $\theta (\lambda-\nu)=0$ (because now $\nu=\beta_{\ell^{\prime}}$).
 \end{proof}

\begin{lemma}\label{tlem2}
For $\lambda\in P^+$ and $\nu\in W_0\omega$, the following explicit formula holds
  \begin{equation}
q_{w_{\lambda-\nu}}(I_{w_{\lambda-\nu}} f)((\lambda-\nu)_+)=
 f(\lambda-\nu) -  \theta (\lambda-\nu ) (1-q_0^{-2}) e_q(\nu) f(\lambda) .
\end{equation}
\end{lemma}
The proof exploits the elementary identities (for $f\in C(P)$, $\lambda\in P$, $j=1,\ldots, n$)
\begin{equation}\label{Ijact}
q_j(I_jf)(\lambda )
 =
 \begin{cases}
 f(s_j\lambda ) =f(\lambda-\alpha_j) &\text{if}\ \langle \lambda ,\alpha_j^\vee\rangle =1 \\
f(\lambda-2\alpha_j) +(1-q_j^2)f(\lambda-\alpha_j) &\text{if}\ \langle \lambda ,\alpha_j^\vee\rangle = 2
\end{cases}
\end{equation}
and $q_j^{-1}(I_j)(s_j\lambda)=f(\lambda )$ if $\langle \lambda ,\alpha_j^\vee\rangle =0$ (cf. Eq. \eqref{stable2}).

\begin{proof}[Proof of Lemma \ref{tlem2}]
The proof  of the lemma employs induction on $\ell (w_{\lambda-\nu})$ starting from the trivial base
$\lambda-\nu\in P^+$. Let $\ell (w_{\lambda-\nu})>1$ and $s_j$ ($1\leq j\leq n$) such that
$$\ell (w_{\lambda-\nu}s_j)=\ell (w_{\lambda-\nu})-1$$ (i.e. $\alpha_j\in R(w_{\lambda-\nu})$).
From the observations following the statement of Lemma \ref{tlem1} it is clear that
$w_{\lambda-\nu}s_j=w_{s_j(\lambda-\nu)}$  with either
$s_j(\lambda -\nu)=\lambda-s_j\nu$ (Cases $(A)$ and $(B)$) or
$s_j(\lambda -\nu)=\lambda (\in P^+)$  (Case $(C)$). In the latter case $w_{\lambda-\nu}=s_j$ and the statement
of the lemma reduces to the first case of  Eq. \eqref{Ijact} (with $\lambda$ replaced by $\lambda-\nu$).
Moreover, in the Cases $(A)$ and $(B)$  invoking of the induction hypothesis yields
\begin{eqnarray}
q_{w_{\lambda-\nu}}(I_{w_{\lambda-\nu}} f)((\lambda-\nu)_+)=
q_{w_{\lambda-s_j\nu}} q_j (I_{w_{\lambda-s_j\nu}} I_jf)((\lambda-s_j\nu)_+) \nonumber \\ =
q_j(I_j f) (\lambda-s_j\nu) - q_j \theta (\lambda-s_j\nu ) (1-q_0^{-2}) e_q(s_j\nu) (I_jf)(\lambda)
\end{eqnarray}
(where we have used that $(\lambda-s_j\nu)_+=(\lambda-\nu)_+$).
In Case $(A)$, one has that $$q_j(I_j f) (\lambda-s_j\nu) =f(\lambda -\nu)$$ (by the first case of Eq. \eqref{Ijact}
with $\lambda$ replaced by $\lambda-s_j\nu$) and $$(I_jf)(\lambda )=q_j f(\lambda )$$ (as $s_j\in W_{0,\lambda}$),
which completes the induction step for this situation upon observing that $\theta (\lambda-s_j\nu)=\theta (\lambda-\nu)$,
$e_q(s_j\nu)=e_q(\nu)q_j^{-2\langle \nu,\alpha_j^\vee\rangle}=e_q(\nu)q_j^{-2}$.
In Case $(B)$ we have that $\theta (\lambda-s_jv)=0$ (since $0\leq\theta (\lambda-s_jv)<\theta (\lambda-\nu)\leq 1$
(cf. Lemma \ref{tlem1})) and $$q_j(I_j f) (\lambda-s_j\nu) =f(\lambda-\nu)-q_0^2(1-q_0^{-2})f(\lambda )$$ (by the second case of Eq. \eqref{Ijact} with $\lambda$ replaced by $\lambda-s_j\nu$ and the fact that
$q_j=q_0$), which completes the induction step for this situation
 upon observing that $\theta (\lambda-\nu)=1$ and
$e_q(\nu)=e_q(\alpha_j)=q_j^2=q_0^2$ (as $e_q(\alpha_j)= e_q(-\alpha_j)
q_j^{2\langle \alpha_j,\alpha_j^\vee\rangle}=e_q(-\alpha_j) q_j^4$).
\end{proof}

We are now in the position to verify Eq. \eqref{qm2} by making the action of the operator on the LHS explicit:
\begin{align*}
(\epsilon ^{-1}\mathcal J  I_{w_o}f)(\lambda-\nu)\stackrel{\text{Eq.}~\eqref{Pop}}{=}&
q_{w_o} q_{w_{\lambda-\nu}}^{-1}(I_{w_{\lambda-\nu}}f)((\lambda-\nu)_+) \\
\stackrel{\text{Eq.}~\eqref{stable2}}{=}& q_{w_o} q_{w_{w_\lambda (\lambda-\nu)}w_\lambda}^{-1}(I_{w_{w_\lambda (\lambda-\nu)}w_\lambda}f)((\lambda-\nu)_+)  .
\end{align*}
For $(\lambda-\nu)_+\neq\lambda_+$,
Lemma \ref{tlem1} (with $\lambda$ and $\nu$ replaced by
$\lambda_+$ and $w_\lambda\nu$) ensures that $w_{w_\lambda (\lambda-\nu)}\in W_{0,\lambda_+}$, whence
$$\ell (w_{w_\lambda (\lambda-\nu)}w_\lambda)=\ell (w_{w_\lambda (\lambda-\nu)})+\ell (w_\lambda)$$
and we may rewrite the expression in question as
\begin{eqnarray*}
\lefteqn{q_{w_o} q_{w_{w_\lambda (\lambda-\nu)}w_\lambda}^{-1}(I_{w_{w_\lambda (\lambda-\nu)}} I_{w_\lambda}f)((\lambda-\nu)_+)
 \stackrel{\text{Lem.}~\ref{tlem2}}{=} }  && \\
 && a_{\lambda ,\nu}^{-1}  q_{w_o}q_{w_\lambda} ^{-1}\left(
(I_{w_\lambda}f)(w_\lambda (\lambda -\nu) )
 -\theta(w_\lambda (\lambda-\nu)) (1-q_0^{-2}) e_q(w_\lambda \nu)
(I_{w_\lambda} f)(\lambda_+)
\right)  ,
\end{eqnarray*}
which proves Eq. \eqref{qm2} when $(\lambda-\nu)_+\neq\lambda_+$.
Similarly, for $(\lambda-\nu)_+= \lambda_+$ we rewrite the expression under consideration as
\begin{eqnarray*}
\lefteqn{ q_{w_o} q_{w_{w_\lambda (\lambda-\nu)}w_\lambda}^{-1}(I_{w_{w_\lambda (\lambda-\nu)}w_\lambda}f)(\lambda_+)
 \stackrel{\text{Eq.}~\eqref{Hrela}}{=} }  && \\
  && q_{w_o} q_{w_{w_\lambda (\lambda-\nu)}w_\lambda}^{-1}\left(
(I_jI_{s_jw_{w_\lambda (\lambda-\nu)} w_\lambda}f)(\lambda_+)  \right. \\
&& \left.  -\chi((s_jw_{w_\lambda (\lambda-\nu)} w_\lambda)^{-1}\alpha_j)  (q_j-q_j^{-1})
(I_{s_jw_{w_\lambda (\lambda-\nu)}w_\lambda} f)(\lambda_+)
\right)   \\
&&= a_{\lambda ,\nu}^{-1}  q_{w_o}q_{w_\lambda} ^{-1}\left(
(I_{w_\lambda}f)(w_\lambda (\lambda -\nu) )
 -\chi(\nu)(1-q_0^{-2})q_{w_{w_\lambda (\lambda-\nu)}}^2
(I_{w_\lambda} f)(\lambda_+)
\right) .
\end{eqnarray*}
In the last step it was used that for $j$ chosen as in Lemma \ref{tlem1} (with $\lambda$ and $\nu$ replaced by
$\lambda_+$ and $w_\lambda\nu$), one has that
$$(s_jw_{w_\lambda (\lambda-\nu)} w_\lambda)^{-1}\alpha_j=\nu ,\quad
s_jw_{w_\lambda (\lambda-\nu)}\in W_{0,\lambda_+},\quad \theta (w_\lambda (\lambda-\nu ))=0,$$ and $q_j=q_0$.
It thus follows for the first term that
\begin{eqnarray*}
\lefteqn{(I_jI_{s_jw_{w_\lambda (\lambda-\nu)} w_\lambda}f)(\lambda_+)=}&& \\
 &&(I_{w_{w_\lambda (\lambda-\nu)} }I_{ w_\lambda}f)((\lambda-\nu)_+) \stackrel{\text{Lem.}~\ref{tlem2}}{=} q_{w_{w_\lambda (\lambda-\nu)} }^{-1}(I_{w_\lambda}f)(w_\lambda (\lambda -\nu) )
 \end{eqnarray*}
and for the second term that
$$(I_{s_jw_{w_\lambda (\lambda-\nu)}w_\lambda} f)(\lambda_+)\stackrel{\text{Eq.}~\eqref{stable2}}{=}
q_{w_{w_\lambda (\lambda-\nu)}}q_0^{-1} (I_{w_\lambda} f)(\lambda_+),$$
where we have exploited that $$\ell (s_jw_{w_\lambda (\lambda-\nu)}w_\lambda)=\ell (s_jw_{w_\lambda (\lambda-\nu)})+\ell(w_\lambda)=
\ell (w_{w_\lambda (\lambda-\nu)})+\ell(w_\lambda)-1.$$
The case $(\lambda-\nu)_+= \lambda_+$ of the identity in Eq. \eqref{qm2} now follows from the fact that
$q_{w_{w_\lambda (\lambda-\nu)}}^2=e_q(w_\lambda \nu)$.
Indeed, for any $w\in W_0$ and $\mu\in P$  one has that
\begin{subequations}
\begin{equation}\label{(i)a}
\langle w^{-1}\mu,\rho^\vee\rangle=\langle \mu,\rho^\vee\rangle+\sum_{\alpha\in R(w)}\langle w^{-1}\mu,\alpha^\vee\rangle
\end{equation} and
\begin{equation}\label{(i)b}
e_q(w^{-1}\mu)=e_q(\mu)\prod_{\alpha\in R(w)}q_\alpha^{2\langle w^{-1}\mu,\alpha^\vee\rangle}
\end{equation}
\end{subequations}
(cf. \cite[\text{Eq.}~(1.5.3)]{mac:affine}),
and
\begin{equation}\label{(ii)}
q_w=\prod_{\alpha\in R(w)} q_\alpha
\end{equation}
(cf. Eq. \eqref{reconstruct}).
Lemma \ref{tlem1} (with $\lambda$ and $\nu$ replaced by $\lambda_+$ and $w_\lambda \nu$) and Property \eqref{(i)a}, \eqref{(i)b} with
$\mu=\alpha_j=-w_{w_\lambda (\lambda-\nu)}w_\lambda\nu$ and $w=s_jw_{w_\lambda (\lambda-\nu)}$ entail that
\begin{equation}\label{eqw}
e_q(w_\lambda\nu)=q_0^2\prod_{\alpha\in R(s_jw_{w_\lambda (\lambda-\nu)})}
q_\alpha^{2\langle w_\lambda\nu,\alpha^\vee\rangle}
\end{equation}
and
\begin{eqnarray}\label{lest}
\ell (w_{w_\lambda(\lambda-\nu)})&\stackrel{\theta (w_\lambda (\lambda-\nu))=0}{=}&\langle w_\lambda\nu,\rho^\vee\rangle \\
& =&1+\sum_{\alpha\in R(s_jw_{w_\lambda (\lambda-\nu)})} \langle w_\lambda\nu,\alpha^\vee\rangle=1+\ell(s_jw_{w_\lambda (\lambda-\nu)}).\nonumber
\end{eqnarray}
Since $\langle w_\lambda\nu,\alpha^\vee\rangle\leq 1$
for $\alpha\in R(s_jw_{w_\lambda (\lambda-\nu)})$ in view of Lemma \ref{tlem1}, it follows from Eq. \eqref{lest} that
in fact $\langle w_\lambda\nu,\alpha^\vee\rangle= 1$
for $\alpha\in R(s_jw_{w_\lambda (\lambda-\nu)})$.  We thus conclude from Eq. \eqref{eqw} that
$$e_q(w_\lambda\nu)=q_0^2\prod_{\alpha\in R(s_jw_{w_\lambda (\lambda-\nu)})}
q_\alpha^{2}\stackrel{\text{Lem.}~\ref{tlem1}}{=}\prod_{\alpha\in R(w_{w_\lambda (\lambda-\nu)})}
q_\alpha^{2} \stackrel{\text{Eq}.~\eqref{(ii)}}{=} q_{w_{w_\lambda (\lambda-\nu)}}^2.$$

\subsection{Proof of Corollary \ref{action-symmetric:cor}}
It is immediate from Theorem \ref{action:thm} that the action of $\widehat{m_\omega (Y)}$ reduces to an
action on
$C(P)^{W_0}\cong C(P^+)$ of the form in Eq. \eqref{Msa} with
\begin{align}\nonumber
V_{\lambda ,-\nu}(q^2)&= q_{t_\lambda}q_{t_{\lambda -\nu}}^{-1}
\sum_{\substack{\nu^\prime\in W_0\omega\\ (\lambda -\nu^\prime)_+=\lambda-\nu}} q_{w_{\lambda-\nu^\prime}}^2
\stackrel{\text{Lem.}~\ref{tlem1}}{=}
e_q(\nu)\sum_{\mu\in W_{0,\lambda} (\lambda -\nu)} q_{w_\mu}^2\\
&=
e_q(\nu) W_{0,\lambda}^{\lambda -\nu}(q^2)=
e_q(\nu)W_{0,\lambda}(q^2)/(W_{0,\lambda}\cap W_{0,\lambda-\nu})(q^2) \label{Valt}
\end{align}
and
\begin{eqnarray}\nonumber
U_{\lambda ,-\omega}(q^2)&=&
\sum_{\substack{\nu\in W_0\omega \\ (\lambda-\nu)_+=\lambda   }} q_{w_{\lambda-\nu}}^2+
(1-q_0^{-2})
\sum_{\nu\in W_0\omega } \varepsilon_{\lambda ,\nu} e_q(\nu)\\
&\stackrel{\text{Lem.}~\ref{tlem1}}{=}&\!\!\!
\sum_{\substack{\nu\in W_0\omega \\ (\lambda-\nu)_+=\lambda   }} q_{w_{\lambda-\nu}}^2+
(1-q_0^{-2})
\sum_{\substack{\nu\in W_0\omega \\ w_{\lambda-\nu} \lambda=\lambda   }} \theta (\lambda -\nu )e_q(\nu) . \label{Ualt}
\end{eqnarray}
This proves Corollary \ref{action-symmetric:cor}
with $V_{\lambda ,-\nu}(q^2)$ and $U_{\lambda ,-\omega}(q^2)$ given by Eqs. \eqref{Valt} and \eqref{Ualt}, respectively.
The coefficient $V_{\lambda ,\nu}(q^2)$ can be recasted in the form given by Eq. \eqref{Msb} upon invoking
Macdonald's product formula \eqref{poincare} and the coefficient $U_{\lambda ,\omega}(q^2)$ can be rewritten
in the form given by Eq. \eqref{Msc} upon comparing the corresponding Pieri formula of the form in Corollary \ref{pieri:cor}
with \cite[\text{Eqs.}~(2.3a)-(2.3c)]{die-ems:pieri}.

\appendix

\section{Braid relation for $A_2$}\label{appA}
In this appendix we verify the braid relation \eqref{br} for the root system $A_2$ via a direct computation. For the root systems $B_2$ and $G_2$ the corresponding computation is analogous (though increasingly tedious).

For $R=A_2$ the braid relation reads:
\begin{subequations}
\begin{equation}\label{A2braid}
\hat T_1 \hat T_2 \hat T_1=    \hat T_2 \hat T_1 \hat T_2 ,
\end{equation}
with
\begin{equation}
\hat T_1=q +\chi_1 (s_1 -1), \qquad \hat T_2=q +\chi_2 (s_2 -1).
\end{equation}
\end{subequations}
Multiplication of the product
\begin{equation*}\label{braid1}
 \hat T_1 \hat T_2= q^2+q\chi_1(s_1-1)+q\chi_2(s_2-1)+\chi_1(s_1-1) \chi_2(s_2-1)
\end{equation*}
from the right by $\hat T_1$ produces
\begin{eqnarray*} \label{braid2}
\lefteqn{\hat T_1 \hat T_2 \hat T_1
=} && \\
&& q^3+2 q^2\chi_1(s_1-1)+q^2\chi_2(s_2-1)+\chi_1(s_1-1) \chi_2(s_2-1) \chi_1(s_1-1)\nonumber \\
&& +q\chi_1(s_1-1)\chi_1(s_1-1)+q\chi_1(s_1-1)\chi_2(s_2-1)+q \chi_2(s_2-1)\chi_1(s_1-1) .\nonumber
\end{eqnarray*}
Swapping the indices $1$ and $2$ yields a corresponding formula for the product $\hat T_2 \hat T_1 \hat T_2$. By comparing both formulas it is seen that the braid relation \eqref{A2braid} amounts to the following identity:
\begin{multline}  \label{braid}
q^2\chi_1(s_1-1) +q \chi_1(s_1-1) \chi_1(s_1-1) +  \chi_1(s_1-1) \chi_2(s_2-1) \chi_1(s_1-1)\\
=
q^2\chi_2(s_2-1) +q \chi_2(s_2-1) \chi_2(s_2-1)  + \chi_2(s_2-1) \chi_1(s_1-1) \chi_2(s_2-1) .
\end{multline}
Upon acting with both sides of Eq. \eqref{braid} on an arbitrary function $f:P\to\mathbb{C}$, it is sufficient to
verify the resulting equality evaluated at the points of a finite $W_0$-invariant set of weights representing the facets of the Coxeter complex for $W_0$ (in view of Lemma \ref{facets:lem}). A convenient choice for such a set of facet representatives
is displayed in Figure \ref{fig1} and the corresponding values confirming the equality of both sides of the identity at these points are collected in Figure \ref{fig2}, where
\begin{align}\label{f0}
f_0:=&f(\omega_1+\omega_2)+f(\omega_1-2\omega_2)+f(-2\omega_1+\omega_2) \\ &-f(-\omega_1-\omega_2)-f(-\omega_1+2\omega_2)-f(2\omega_1-\omega_2).\nonumber
\end{align}

\begin{figure}
\begin{center}
\includegraphics[scale=0.7]{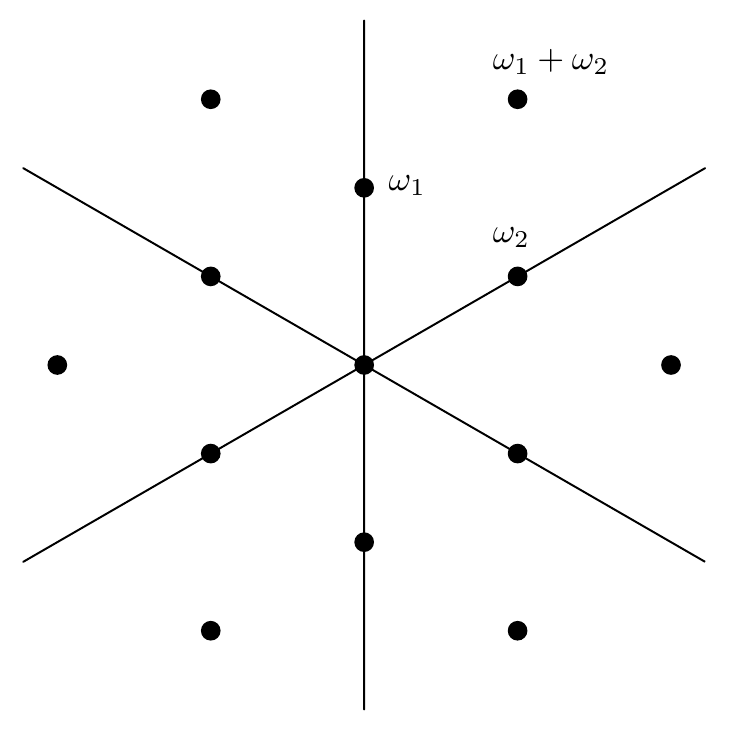}
\end{center}
\caption{The set $W_0\{0,\omega_1,\omega_2,\omega_1+\omega_2\}$ consisting of 13 weights representing the facets of the Coxeter complex of $W_0$ for $R=A_2$.}\label{fig1}
\end{figure}

\begin{figure}\centering
\begin{tabular}{c|c}
$\lambda$ & $\text{LHS}=\text{RHS}$ \\ \hline
 $0, \pm \omega_1, \pm\omega_2,\pm(\omega_1-\omega_2)$  & $0$ \\
 $\omega_1+\omega_2$  & $-q^{-3}f_0$ \\
 $-\omega_1-\omega_2$  &  $q^3 f_0$\\
$\omega_1-2\omega_2$,  $-2\omega_1+\omega_2$    &  $-qf_0$\\
 $-\omega_1+2\omega_2$,  $2\omega_1-\omega_2$   & $q^{-1}f_0$
\end{tabular}
\caption{Values of both sides of Eq. \eqref{braid} upon acting on an arbitrary function
$f:P\to\mathbb{C}$ and evaluation at the points $\lambda$ of $W_0\{0,\omega_1,\omega_2,\omega_1+\omega_2\}$. (Here $f_0$ is given by Eq. \eqref{f0}.)}\label{fig2}
\end{figure}

\section{Affine intertwining relations}\label{appB}
In this appendix we prove the affine intertwining relations in Lemma \ref{affine-intertwining:lem} (therewith completing the proof of the intertwining property in Theorem \ref{intertwining:thm}).

\subsection{Preparations: some properties related to (quasi-)minuscule weights}
The proof of the affine intertwining relations is based on properties of certain special elements in $W$ and $\mathcal{H}$ associated with the minuscule and quasi-minuscule weights. Let us recall in this connection that the minuscule weights $\omega$ are characterized by the property that $0\leq \langle \omega,\alpha^\vee\rangle \leq 1$ for all $\alpha\in R^+$, whereas the quasi-minuscule weight $\omega =\alpha_0$ is characterized
by the property that
$0\leq\langle \omega,\alpha^\vee\rangle \leq 2$ for all $\alpha\in R^+$ with the upper bound $2$ being reached only {\em once} (viz. for $\alpha=\alpha_0$).

\begin{lemma}\label{weights:lem}
Let $\mu\in P$.
\begin{itemize}
\item[(i)] If $\omega\in P^+$ is minuscule, then $\mu_+ +w_\mu\omega\in P^+$.
\item[(iia)]  If $\mu_+ +w_\mu \alpha_0\not\in P^+$, then $w_\mu\alpha_0=-\alpha_j$ for some $1\leq j\leq n$ and moreover
$\langle \mu_+,\alpha_j^\vee\rangle =1$.
\item[(iib)] If $\mu_+ +w_\mu \alpha_0\in P^+$ with $w_\mu \alpha_0\in R^-$, then $\langle \mu_+ ,w_\mu\alpha_0^\vee\rangle \leq -2$.
\end{itemize}
\end{lemma}
\begin{proof}
(i) For any $1\leq j\leq n$,  one has that
$\langle \mu_+ +w_\mu\omega,\alpha_j^\vee \rangle\geq \langle \mu_+ ,\alpha_j^\vee\rangle-1\geq -1$.  The statement now amounts to the observation that lower bound $-1$ cannot be reached. Indeed,
if $\langle \mu_+ ,\alpha_j^\vee \rangle =0$ then $s_j\in W_{\mu_+}$, whence
$\ell (s_jw_\mu )=\ell (w_\mu)+1$, i.e. $w_\mu^{-1}\alpha_j\in R^+$, and thus $
 \langle \mu_+ +w_\mu\omega,\alpha_j^\vee \rangle\geq \langle w_\mu\omega,\alpha_j^\vee \rangle\geq 0$.

(iia) By definition  the assumption implies that
$\langle \mu_+ +w_\mu \alpha_0,\alpha^\vee_j\rangle  <0$ for some $1\leq j\leq n$. Hence
$\langle w_\mu \alpha_0,\alpha^\vee_j \rangle =-2$  with $0\leq\langle \mu_+,\alpha_j^\vee\rangle \leq 1$ or
$\langle w_\mu \alpha_0,\alpha^\vee_j \rangle =-1$  with $\langle \mu_+,\alpha_j^\vee\rangle =0$. By repeating the argument of part (i), it is seen that $\langle \mu_+,\alpha_j^\vee\rangle $ cannot be zero.  It thus follows that
$\langle \mu_+,\alpha_j^\vee\rangle =1$ and that $\langle w_\mu \alpha_0,\alpha^\vee_j \rangle =-2$, i.e.
$w_\mu \alpha_0=-\alpha_j$.

(iib)  Immediate from the estimate $\langle \mu_+,w_\mu\alpha_0^\vee\rangle=
\langle \mu_+ +w_\mu\alpha_0,w_\mu\alpha_0^\vee\rangle-2\leq -2$.
\end{proof}

\begin{lemma}[\cite{mac:affine}]\label{hrel:lem}
Let $w\in W_0$.
\begin{itemize}
\item[(i)]  If $\omega\in P^+$ is minuscule, then
\begin{equation*}
T_w^{-1} Y^\omega T_{v_\omega}^{-1}=Y^{w^{-1}\omega } T^{-1}_{v_\omega w}.
\end{equation*}
\item[(ii)]  If $s=s_{\alpha_0,0}$, then
 \begin{equation*}
 T^{-1}_w T_0^{\text{sign} (w^{-1}\alpha_0)}  =Y^{w^{-1}\alpha_0}T^{-1}_{sw}.
 \end{equation*}
\end{itemize}
\end{lemma}
\begin{proof}
Both relations are a consequence of \cite[(3.3.2)]{mac:affine}. More specifically, (i) and
(ii) amount to {\em loc. cit.}  (3.3.3) and (3.3.6), respectively.
\end{proof}

\subsection{Proof of  $u \mathcal{J}=\mathcal{J} I_u$}
It is sufficient to verify the intertwining relation for $u=u_\omega$ with $\omega\in P^+$ minuscule (cf. Eq. \eqref{finite-group}).

Let $f\in C(P)$ and let $\mu,\omega\in P$ with $\omega$ minuscule. By definition, we have that
$$(u_\omega \mathcal{J}f)(w_o\mu)=q_{t_\nu}q_{w_\nu} (I^{-1}_{w_\nu^{-1}}f)(\nu_+)\quad
\text{with}\quad \nu :=u_\omega^{-1}w_o\mu .$$
Upon setting $\tilde{w}:=w_\mu w_o v_\omega^{-1}$, it is readily seen that
$\tilde{w} \nu =\mu_+ +w_\mu\omega^* =\nu_+$
in view of Lemma \ref{weights:lem} part (i). Hence, invoking of Eq. \eqref{stable} infers that
$$(u_\omega \mathcal{J}f)(w_o\mu)=q_{t_\nu}q_{\tilde{w}} (I^{-1}_{\tilde{w}^{-1}}f)(\nu_+).$$
Similarly, we have that
\begin{align*}
(\mathcal{J}I_{u_\omega}f)(w_o \mu)&=  (\mathcal{J}t_\omega I^{-1}_{v_\omega}f)(w_o \mu)\\
&=
 q_{t_{w_o\mu}} q_{w_{w_o\mu}} (I^{-1}_{w^{-1}_{w_o\mu}} t_\omega I^{-1}_{v_\omega}f)(\mu_+)\\
 &= q_{t_{\mu}} q_{w_{\mu}w_o} (I^{-1}_{w_ow^{-1}_{\mu}} t_\omega I^{-1}_{v_\omega}f)(\mu_+)
\end{align*}
(where in the last step we have again applied Eq. \eqref{stable}).
The stated equality now follows because
$$
q_{t_\nu}q_{\tilde{w}} = q_{t_{\mu_++w_\mu\omega^*}} q_{v_\omega w_o w_\mu^{-1}}\stackrel{(i)}{=}
q_{t_{\mu_+}}q_{w_ow_\mu^{-1}}= q_{t_{\mu}} q_{w_{\mu}w_o}
$$
and
$$
(I^{-1}_{\tilde{w}^{-1}}f)(\nu_+)=   (t_{w_\mu w_o \omega} I^{-1}_{v_\omega w_ow_\mu^{-1}}f)(\mu_+) \stackrel{(ii)}{=}
(I^{-1}_{w_ow^{-1}_{\mu}} t_\omega I^{-1}_{v_\omega}f)(\mu_+)  ,
$$
where in steps $ (i)$ and $(ii)$ we relied
on the relation $$v_\omega w_o w_\mu^{-1} t_{\mu_++w_\mu\omega^*} =
u_\omega^{-1} w_ow_\mu^{-1}t_{\mu_+} $$ (with $u_\omega\in\Omega$) and Lemma \ref{hrel:lem} part (i) with $w=w_ow_\mu^{-1}$, respectively.

\subsection{Proof of  $\hat{T}_0  \mathcal{J}=\mathcal{J} I_0$}
Let $f\in C(P)$ and let $\mu\in P$.
By definition (and application of Eq. \eqref{stable}) it is immediate that
\begin{eqnarray}\label{LHS}
\lefteqn{(\hat{T}_0 \mathcal{J} f)(w_o \mu)=
q_0 q_{t_\mu}q_{w_\mu w_o} (I^{-1}_{w_ow_\mu^{-1}}f)(\mu_+) }  && \\
&&+  \chi_0 (w_o\mu )
\left( q_{t_\nu}q_{w_\nu} (I^{-1}_{w_\nu^{-1}}f)(\nu_+)
-q_{t_\mu}q_{w_\mu w_o} (I^{-1}_{w_ow_\mu^{-1}}f)(\mu_+)
\right)  \nonumber
\end{eqnarray}
and $$(\mathcal{J}I_0f)(w_o\mu)=q_{t_\mu}q_{w_\mu w_o} (I^{-1}_{w_ow_\mu^{-1}}I_0 f)(\mu_+)$$ ($=
q_{t_\mu}q_{w_\mu w_o} (I^{-1}_{w_ow_\mu^{-1}}t_{\alpha_0}I^{-1}_s f)(\mu_+)$), with $\nu: = s_0w_o\mu=sw_o(\mu+\alpha_0)$  and
$s:=s_{\alpha_0,0}$, respectively. We will distinguish three disjoint situations.

{\em Case (A):} $\mu_++w_\mu\alpha_0\not\in P^+$.
By Lemma \ref{weights:lem} part (iia) we have in this case that $w_\mu\alpha_0=-\alpha_j$ with $\langle \mu_+,\alpha_j^\vee\rangle =1$ for some $1\leq j\leq n$.  But then $q_0=q_j$ and $w_o\mu\in V_0$, i.e. $s_0(w_o\mu )=w_o\mu$,  $\chi_0(w_o\mu)=1$, $\nu_+=\mu_+$. The stated equality thus reduces to
$$(I^{-1}_{w_ow_\mu^{-1}} I_0 f)(\mu_+)=q_0(I^{-1}_{w_ow_\mu^{-1}}f)(\mu_+)$$ (because the terms within the bracket on the second line of Eq. \eqref{LHS} now cancel each other by Eq. \eqref{stable}). In view of Lemma \ref{hrel:lem} part (ii) (with $w=w_ow_\mu^{-1}$) this amounts to the equation
$(t_{\alpha_j}I^{-1}_{sw_ow_\mu^{-1}}f)(\mu_+)=q_0 ( I^{-1}_{w_ow_\mu^{-1}}f)(\mu_+)$. Since
$sw_ow_\mu^{-1}=w_ow_\mu^{-1}s_j$, $\ell (w_ow_\mu^{-1}s_j)=\ell (w_ow_\mu^{-1})+1$, and $q_0=q_j$, the latter equation can be rewritten as $$(t_{\alpha_j}I^{-1}_jg)(\mu_+)=q_jg(\mu_+)\quad \text{with}\quad g:=I^{-1}_{w_ow_\mu^{-1}}f.$$ This last equality is immediate (for any $g\in C(P)$) from the quadratic relation $I_j^{-1}=I_j-(q_j-q_j^{-1})$ together with the definition of $I_j$ (taking into account that
$s_j\mu_+=t_{\alpha_j}^{-1}\mu_+$).

From now on we will assume that $\mu_++w_\mu\alpha_0\in P^+$ (i.e. we are not in Case (A)).
Then---upon setting $\tilde{w}:=w_\mu w_os$---it is clear that
$\tilde{w}\nu=\mu_++w_\mu\alpha_0=\nu_+$. Hence, application of Eq. \eqref{stable}  allows us to rewrite
$q_{w_\nu} (I^{-1}_{w_\nu^{-1}}f)(\nu_+)$ as
$$q_{\tilde{w}} (I^{-1}_{\tilde{w}^{-1}}f)(\nu_+)= q_{\tilde{w}} (t_{w_\mu w_o\alpha_0}I^{-1}_{sw_ow_\mu^{-1}}f)(\mu_+).$$  Combining this with the relation
\begin{equation}\label{length-rel}
q_{\tilde{w}}q_{t_\nu}=q_{w_\mu w_o}q_{t_\mu}q_0^{\text{sign}(w_\mu\alpha_0)}
\end{equation}
(proven below) and division by common factors turns Eq. \eqref{LHS} into
\begin{eqnarray}
\lefteqn{q_{t_\mu}^{-1}q_{w_\mu w_o}^{-1}(\hat{T}_0 \mathcal{J} f)(w_o \mu)=q_0  (I^{-1}_{w_ow_\mu^{-1}}f)(\mu_+)  } && \\
&&
+  \chi_0 (w_o\mu )
\left(  q_0^{\text{sign}(w_\mu\alpha_0)} (t_{w_\mu w_o\alpha_0}I^{-1}_{sw_ow_\mu^{-1}}f)(\mu_+)-(I^{-1}_{w_ow_\mu^{-1}}f)(\mu_+)
\right) , \nonumber
\end{eqnarray}
which must now be shown to coincide with  $(I^{-1}_{sw_ow_\mu^{-1}}I_0 f)(\mu_+)$.

{\em Case (B):}  $w_\mu\alpha_0\in R^-$ (and $\mu_++w_\mu\alpha_0\in P^+$).
By Lemma \ref{weights:lem} part (iib) we have in this case that $\langle \mu_+, w_\mu\alpha_0^\vee \rangle \leq -2$, whence
$\chi_0 (w_o\mu )=q_0$. The stated equality therefore reduces to
$$(t_{w_\mu w_o\alpha_0}I^{-1}_{sw_ow_\mu^{-1}}f)(\mu_+)=
 (I^{-1}_{w_ow_\mu^{-1}}I_0 f)(\mu_+),$$
which follows from Lemma \ref{hrel:lem} part (ii) (with $w=w_ow_\mu^{-1}$).

{\em Case (C):}  $w_\mu\alpha_0\in R^+$ (and $\mu_++w_\mu\alpha_0\in P^+$).
Now $\chi_0(w_o\mu)=q_0^{-1}$, whence the stated equality becomes
$$(q_0-q_0^{-1})(I^{-1}_{w_ow_\mu^{-1}}f)(\mu_+)+(t_{w_\mu w_o \alpha_0} I^{-1}_{sw_ow_\mu^{-1}}f)(\mu_+)
=(I^{-1}_{w_ow_\mu^{-1}}I_0f)(\mu_+).$$ Applying the quadratic relation $I_0=I_0^{-1}+(q_0-q_0^{-1})$ rewrites this as $$(t_{w_\mu w_o \alpha_0} I^{-1}_{sw_ow_\mu^{-1}}f)(\mu_+)=(I^{-1}_{w_ow_\mu^{-1}}I_0^{-1}f)(\mu_+),$$ which again follows by Lemma \ref{hrel:lem} part (ii) (with $w=w_ow_\mu^{-1}$).

It remains to verify the above relation in Eq. \eqref{length-rel} for the length multiplicative function. Indeed, straightforward
manipulations reveal that
\begin{eqnarray*}
\lefteqn{q_{\tilde{w}}q_{t_\nu}=q_{sw_ow_\mu^{-1}}q_{t_{\mu_++w_\mu\alpha_0}}
\stackrel{(i)}{=}q_{s_0w_ow_\mu^{-1}t_{\mu_+}}}&& \\
&& \stackrel{(ii)}{=}q_{w_ow_\mu^{-1}}q_{t_{\mu_+}}q_0^{\text{sign}(w_\mu\alpha_0)}=
q_{w_\mu w_o}q_{t_{\mu}}q_0^{\text{sign}(w_\mu\alpha_0)},
\end{eqnarray*}
where in steps  $(i)$ and $(ii)$ we relied on the elementary relations $$s_0w_ow_\mu^{-1}t_{\mu_+}=sw_ow_\mu^{-1}t_{\mu_++w_\mu\alpha_0}$$ and $$ \ell (s_0w_ow_\mu^{-1}t_{\mu_+})=\ell (w_ow_\mu^{-1}t_{\mu_+})+\text{sign}(w_\mu\alpha_0),$$ respectively.

\section{Explicit formulas for $R=A_{N-1}$}\label{appC}

In this appendix we exhibit explicit formulas describing the differential-reflection representation and the integral-reflection representation for the root system $A_{N-1}$, as well as the corresponding discrete difference operators diagonalized by the Hall-Littlewood polynomials. To facilitate their direct use in the theory of symmetric functions it will be convenient to employ a central extension of the $A_{N-1}$-type extended affine Weyl group and its Hecke algebra (associated with $GL_N$ rather than $SL_N$).

\subsection{Affine permutation group}
For $R=A_{N-1}$ the finite Weyl group amounts to the permutation group $S_N$ and the
corresponding extended affine Weyl group is given by the affine permutation group
$W=S_N\ltimes \mathbb{Z}^N$, which acts
on $\mathbb{R}^N$ by permuting the elements of the standard basis $e_1,\ldots ,e_N$ and
translating over vectors in the integral lattice, i.e.
for $w\in S_N$, $\lambda=(\lambda_1,\ldots ,\lambda_N)\in\mathbb{Z}^N$ and
$x=(x_1,\ldots ,x_N)\in \mathbb{R}^N$:
\begin{subequations}
\begin{align}
wx& =(x_{w^{-1}(1)},\ldots ,x_{w^{-1}(N)}) ,\\
t_\lambda x&=(x_1+\lambda_1,\ldots,x_N+\lambda_N) .
\end{align}
\end{subequations}
The group $W$ is generated by the finite transpositions
\begin{subequations}
\begin{equation}
s_jx=(x_1,\ldots ,x_{j-1},x_{j+1},x_j ,x_{j+2},\ldots ,x_N)\qquad (1\leq j< N)
\end{equation}
and the affine generator
\begin{equation}
u x=(x_N+1,x_1,\ldots,x_{N-1})
\end{equation}
\end{subequations}
(so $u^N=t_{e_1+\cdots +e_n}$ lies in the center of $W$).
The translations over the vectors of the standard basis can be expressed in terms of these generators as:
\begin{equation}  \label{reducedtej}
t_{e_j}=s_{j-1}\cdots s_2 s_1  u s_{N-1}s_{N-2}\cdots s_j \qquad (1\leq j\leq N).
\end{equation}

\subsection{Affine Hecke algebra}
The extended affine Hecke algebra $\mathcal{H}$ associated with $W$ is the complex associative algebra generated by
the invertible elements $T_1,\dots, T_{N-1}$ and $T_u$ subject to the relations:
\begin{subequations}
\begin{align}
(T_j-q)&(T_j+q^{-1}) =0 \qquad \quad    (1\leq j< N) ,\\
T_j T_k&= T_k T_j  \qquad \quad      ( 1\leq j<k-1< N-1),\\
T_j T_{j+1}T_{j+1} &= T_{j+1} T_j T_{j+1}  \qquad   (1\leq j< N-1),\\
T_u T_j &=T_{j+1} T_u\qquad \qquad   (1\leq j< N-1),\\
T_u^NT_j&=T_jT_u^N \qquad \qquad \quad  (1\leq j< N).
  \end{align}
\end{subequations}

The Bernstein-Lusztig-Zelevinsky basis for $\mathcal{H}$ is in this situation of the form $T_wY^\lambda$ ($w\in S_N$, $\lambda\in\mathbb{Z}^N$), where $T_w:=T_{s_{j_1}}\cdots T_{s_{j_\ell}}$ for $w=s_{j_1}\cdots s_{j_\ell}$
a reduced expression ($\ell =\ell (w)$) and $Y^\lambda:=Y_1^{\lambda_1} \dots Y_N^{\lambda_N}$ with (cf. Eq. \eqref{reducedtej})
\begin{equation}
Y_j:=T_{j-1}^{-1}\cdots T_2^{-1}T_1^{-1} T_u T_{N-1} T_{N-2}\cdots T_{j+1}T_j\qquad (1\leq j\leq N)
\end{equation}
pairwise commutative.

\subsection{Difference-reflection representation} For $1\leq j <N$, let
$\hat{T}_j:C(\mathbb{Z}^N)\to C(\mathbb{Z}^N)$ be defined as
\begin{equation}
(\hat T_j f)(\lambda)= \\
\begin{cases}
(q-q^{-1})f(\lambda)  + q^{-1}f(s_j\lambda)   & \text{if $\lambda_j> \lambda_{j+1}$}\\
qf(\lambda)   & \text{if $\lambda_j= \lambda_{j+1}$} \\
 qf(s_j\lambda)   & \text{if $\lambda_j< \lambda_{j+1}$}\\
\end{cases}
\end{equation}
($f\in C(\mathbb{Z}^N)$, $\lambda\in\mathbb{Z}^N$). The difference-reflection representation
$h\mapsto \hat{T}(h)$ ($h\in \mathcal H$) of the extended affine Hecke algebra on $\mathcal C(\mathbb Z^N)$
is determined by the assignment
$T_j\mapsto \hat{T}_j$ ($1\leq j< N$) and $T_u\mapsto u$.

\subsection{Integral-reflection representation}
For $1\leq j <N$, let
$I_j:C(\mathbb{Z}^N)\to C(\mathbb{Z}^N)$ be defined as
\begin{eqnarray}
\lefteqn{(I_jf)(\lambda)=q f(s_j\lambda)+(q-q^{-1})\times} && \\
&&
\begin{cases}
\displaystyle
-\sum_{l=1}^{\lambda_j-\lambda_{j+1}}  f(\lambda_1,\dots, \lambda_j-l, \lambda_{j+1}+l,\dots, \lambda_N) & \text{if $\lambda_j> \lambda_{j+1}$}\\
\qquad\qquad\qquad 0 & \text{if $\lambda_j= \lambda_{j+1}$} \\
\displaystyle
\sum_{l=0}^{\lambda_{j+1}-\lambda_j-1}  f(\lambda_1,\dots, \lambda_j+l, \lambda_{j+1}-l,\dots, \lambda_N)  & \text{if $\lambda_j< \lambda_{j+1}$}
\end{cases} \nonumber
\end{eqnarray}
($f\in C(\mathbb{Z}^N)$, $\lambda\in\mathbb{Z}^N$).
The integral-reflection representation
$h\mapsto I(h)$ ($h\in \mathcal H$) of the extended affine Hecke algebra on $\mathcal C(\mathbb Z^N)$
is determined by the assignment
$T_j\mapsto I_j$ ($1\leq j< N$) and $Y^\lambda\mapsto t_\lambda$ ($\lambda\in\mathbb{Z}^N$).

\subsection{Central difference operators}
The elementary symmetric polynomials
\begin{equation}
m_r (Y):=\sum_{\substack{J\subset \{ 1,\ldots ,N\} \\ |J|=r}} \prod_{j\in J} Y_j\qquad (r=1,\ldots ,N)
\end{equation}
lie in the center of $\mathcal{H}$.  The explicit action in $C(\mathbb{Z}^N)$ of the corresponding operators $\widehat{m_1(Y)},\ldots ,\widehat{m_N(Y)}$ under the difference-reflection representation is of the form:
\begin{subequations}
\begin{equation}
\widehat{m_r(Y)}=\epsilon M_r \epsilon^{-1}\qquad (r=1,\ldots ,N),
\end{equation}
with $\epsilon:\mathcal C(\mathbb Z^N)\to \mathcal C(\mathbb Z^N)$ and  $M_r:\mathcal C(\mathbb Z^N)\to \mathcal C(\mathbb Z^N)$ given by
\begin{align}
(\epsilon f)(\lambda_1,\dots, \lambda_N)&=q^{2\langle \rho,\lambda_+\rangle}f(\lambda_N,\lambda_{N-1},\dots, \lambda_1) , \\
(M_r f)(\lambda)&=\sum_{\substack{J\subset\{1,2,\dots, N\} \\  |J|=r}} q^{2\ell(w_{w_\lambda(\lambda-e_J)})} f(\lambda-e_J)
\end{align}
\end{subequations}
($f\in C(\mathbb{Z}^N)$, $\lambda\in\mathbb{Z}^N$).
Here  $\rho:=\frac{1}{2}(N-1,N-3,\dots, 3-N, 1-N)$,
$\lambda_+$ is obtained from $\lambda$ by reordering the components of $\lambda$ in (weakly) decreasing order, $w_\lambda$ denotes the shortest permutation in $S_N$ taking $\lambda$ to $\lambda_+$,
and $e_{J}:=\sum_{j\in J} e_j$.

The restriction of the action of $\widehat{m_r(Y)}$ to $C(\mathbb{Z}^N)^{S_N}\simeq C(\mathbb{Z}^N_{\geq})$ with
\begin{equation*}
\mathbb{Z}^N_{\ge}:=\{ \lambda\in \mathbb{Z}^N\mid \lambda_1\geq \dots \geq \lambda_N\}
\end{equation*}
is given by
\begin{subequations}
\begin{eqnarray}
\lefteqn{(\widehat{m_r(Y)}f)(\lambda)=} &&\\
&& \sum_{\substack{J\subset\{1,2,\dots, N\},\,  |J|=r \\ \lambda-e_J\in \mathbb Z^N_{\ge} }}
V_{\lambda,J^c}(q^2) f(\lambda-e_J),
\quad (f\in \mathcal C(\mathbb Z^N_\ge), \lambda\in \mathbb Z^N_\ge) , \nonumber
\end{eqnarray}
where $J^c:=\{1,\dots, N\}\backslash J$ and
\begin{equation}
V_{\lambda, J}(q^2):= q^{-2\langle \rho,e_J\rangle}\prod_{\substack{1\le k<l\le N \\ k\in J, l\in J^c\\ \lambda_k=\lambda_l}}\frac{1-q^{2(l-k+1)}}{1-q^{2(l-k)}} .
\end{equation}
\end{subequations}

The diagonal action of $\widehat{m_r(Y)}$ on the Hall-Littlewood basis
entails the following Pieri formula:
\begin{equation}
m_r p_\lambda=\sum_{\substack{J\subset\{1,2,\dots, N\},\,  |J|=r \\ \lambda+e_J\in \mathbb Z^N_{\ge} }}
V_{\lambda,J}(q^2) p_{\lambda+e_J}, \quad (r=1,\dots, N)
\end{equation}
for the Hall-Littlewood polynomials
\begin{equation}
p_\lambda=q^{2\langle \rho,\lambda\rangle}\sum_{w\in S_N} x^{w\lambda} \prod_{1\leq k<l\leq N}\frac{x_{wk}-q^2x_{wl}}{x_{wk}-x_{wl}}
\quad (\lambda\in \mathbb Z^N_\ge),
\end{equation}
where $x^\mu:=x_1^{\mu_1}\cdots x_N^{\mu_N}$ ($\mu\in\mathbb{Z}^N$).
This Pieri formula amounts to a classic Pieri formula for the Hall-Littlewood polynomials due to Morris
\cite{mor:note}.

\vspace{3ex}
\noindent {\bf Acknowledgments.} We thank the referees for suggesting some improvements concerning the presentation.
Our verification of the braid relations in Eq. \eqref{br} for the rank-two root systems---via a direct computer assisted computation---benefited a lot  from
Stembridge's Maple packages {\tt COXETER} and
{\tt WEYL}.

\bibliographystyle{amsplain}

\end{document}